\documentclass[a4paper, 11pt]{amsart}
\usepackage{amsmath,amsthm,amsfonts,amscd,eucal,mathtools,mathabx,amssymb,mathrsfs,pictex}
\numberwithin{equation}{section}
\usepackage{tikz}
\usetikzlibrary{cd}

\usepackage{scalerel} %Ridimensionamento simboli
\usepackage{upgreek}

\usepackage{xypic}
\usepackage{graphicx}
\usepackage{stackrel}

\newtheorem{theorem}{Theorem}[section]
\newtheorem{thm}{Theorem}[section]
\newtheorem{lem}[theorem]{Lemma}
\newtheorem{cor}[theorem]{Corollary}
\newtheorem{prop}[theorem]{Proposition}

\newtheorem{defin}[theorem]{Definition}

\theoremstyle{definition}

\newtheorem{rem}[theorem]{Remark}

\makeatletter
\newcommand\smallw{\scaleobj{0.75}{w}} %Ridimensiono w.
\newcommand\smallt{\scaleobj{0.75}{u}} %Ridimensiono u.

\newcommand\smallui{\scaleobj{0.55}{u_1}} %Ridimensiono u_1.
\newcommand\smalluii{\scaleobj{0.55}{u_2}} %Ridimensiono u_2.
\newcommand\smallund{\scaleobj{0.5}{u_{\text{nd}}}} %Ridimensiono u_{\text{nd}}.

\newcommand\smallv{\scaleobj{0.75}{v}} %Ridimensiono v.
\newcommand\smallf{\scaleobj{0.6}{\mathtt{F}}} %Ridimensiono \mathtt{F}.
\newcommand\smallk{\scaleobj{0.6}{\mathtt{K}}} %Ridimensiono \mathtt{K}.

\newcommand\smallbigcirc[1]{\vcenter{\hbox{\scalebox{0.84}{$\m@th#1\bigcirc$}}}}
\newcommand\make@circled[2]{\ooalign{$\m@th#1\smallbigcirc{#1}$\cr\hidewidth\raise.2ex\hbox{$\m@th#1#2$}\hidewidth\cr}}
\newcommand\make@circledvar[2]{\ooalign{$\m@th#1\smallbigcirc{#1}$\cr\hidewidth\raise.4ex\hbox{$\m@th#1#2$}\hidewidth\cr}} %Sovrappongo \smallbigcirc con \smallund.

\newcommand\ow{\mathbin{\mathpalette\make@circled\smallw}}
\newcommand\ot{\mathbin{\mathpalette\make@circled\smallt}} %Definisco \ot con la corretta spaziatura di un'operazione binaria tramite \mathbin.

\newcommand\of{\mathbin{\mathpalette\make@circled\smallf}} %Definisco \of con la corretta spaziatura di un'operazione binaria tramite \mathbin.

\newcommand\ok{\mathbin{\mathpalette\make@circled\smallk}} %Definisco \ok con la corretta spaziatura di un'operazione binaria tramite \mathbin.

\newcommand\ov{\mathbin{\mathpalette\make@circled\smallv}} %Definisco \ov con la corretta spaziatura di un'operazione binaria tramite \mathbin.

\newcommand\oui{\mathbin{\mathpalette\make@circled\smallui}} %Definisco \oui con la corretta spaziatura di un'operazione binaria tramite \mathbin.

\newcommand\ouii{\mathbin{\mathpalette\make@circled\smalluii}} %Definisco \ouii con la corretta spaziatura di un'operazione binaria tramite \mathbin.

\newcommand\ound{\mathbin{\mathpalette\make@circledvar\smallund}} %Definisco \ound con la corretta spaziatura di un'operazione binaria tramite \mathbin.

\makeatother

\def\cb{{\mathcal B}}

\def\ce{{\mathcal E}}

\def\ch{{\mathcal H}}
\def\ci{{\mathcal I}}

\def\cam{{\mathcal M}}

\def\cs{{\mathcal S}}

\def\cu{{\mathcal U}}

\def\ga{{\mathfrak A}}
\def\gb{{\mathfrak B}}
\def\gc{{\mathfrak C}}
\def\gd{{\mathfrak D}}

\def\ba{{\mathbb A}}
\def\bb{{\mathbb B}}
\def\bc{{\mathbb C}}
\def\bbf{{\mathbb F}}

\def\bm{{\mathbb M}}
\def\bn{{\mathbb N}}
\def\bp{{\mathbb P}}

\def\br{{\mathbb R}}
\def\bt{{\mathbb T}}

\def\bz{{\mathbb Z}}

\def\a{\alpha}
\def\b{\beta}
\def\g{\gamma}  \def\G{\Gamma}
  \def\D{\Delta}

\def\la{\lambda} 
\def\k{\kappa}
\def\m{\mu}

\def\n{\nu}

\def\s{\sigma} 
\def\t{\tau}
\def\f{\varphi}  \def\F{\Phi}
\def\th{\theta} 
\def\om{\omega}

\def\ker{\hbox{Ker}}
\def\ker{\mathop{\rm ker}}

\def\slim{\mathop{s\text{-lim}}}

\def\1{\mathds{1}}

\def\wlim{\mathop{\rm w-lim}}
\def\ulim{\mathop{\rm unif-lim}}

\newcommand{\ty}[1]{\mathop{\rm {#1}}}
\def\di{{\rm d}}

\def\re{\mathop{\rm Rep}}

\def\slim{\mathop{\rm s-lim}}
\def\idd{{1}\!\!{\rm I}}

\def\min{\mathop{\rm min}}
\def\max{\mathop{\rm max}}

\DeclareMathAlphabet{\mathpzc}{OT1}{pzc}{m}{it}

\begin{document}

\title[twisted tensor product]
{Infinite twisted $C^*$-tensor product and symmetric states}
\author{Francesco Fidaleo}
\address{Francesco Fidaleo\\
Dipartimento di Matematica \\
Universit\`{a} di Roma Tor Vergata\\
Via della Ricerca Scientifica 1, Roma 00133, Italy} \email{{\tt
fidaleo@mat.uniroma2.it}}
\author{Elia Vincenzi}
\address{Elia Vincenzi\\
Dipartimento di Matematica \\
Universit\`{a} di Roma Tor Vergata\\
Via della Ricerca Scientifica 1, Roma 00133, Italy} \email{{\tt
eliavincenzi@gmail.com}}

\date{\today}

\begin{abstract}
The twisted $C^*$-tensor product was exhaustively investigated by the authors in the previous papers \cite{FV, FV1}. In this new context, generalising the usual tensor product and the Fermi one, we analyse the possibility of studying the set of symmetric states, that is those invariant under all finite permutations, in the setting of infinite twisted $C^*$-tensor products. As a preliminary result, we recognise that such an investigation can proceed only when the bicharacter, involved in the construction of such twisted tensor products, is hermitian. Otherwise, there is no natural action of the finitary symmetric group on the infinite twisted chain. Furthermore, even if the investigation of symmetric states is certainly meaningful for all twisted systems based on hermitian bicharacters, it is shown that it can be fruitfully carried out in three cases only. In these cases, we can provide the ``genuine" version of the celebrated De Finetti Theorem (cf. \cite{DeF}) already established in \cite{HS} for the general classical case, in \cite{St2} for the usual tensor product and in \cite{Ffe} for Fermi models. Very  surprisingly, it emerges that one more twisted model can be treated exhaustively: it corresponds to the chain twisted by the so-called Klein four-group, and its Klein bicharacter unique up to equivalence.

\vskip 0.3cm 
\noindent{\bf Mathematics Subject Classification}: 46L53, 46L05, 60G09, 46L30, 46N50.\\
{\bf Key words}: Noncommutative probability and statistics;
$C^{*}$--algebras, states; Exchangeability;  Klein-Wigner transformation; Klein four-group;
Applications to quantum physics.
\end{abstract}
\maketitle

\tableofcontents

\section{Introduction}
\label{sec1}

The present paper is the last part of a collaboration between the authors, which led to the PhD dissertation of the second one (cf. \cite{V}), and was aimed at carrying out a detailed investigation of the so-called {\it twisted tensor products}. The general part of such a construction is contained in our previous papers \cite{FV, FV1}, whereas this one is devoted to construct the infinite twisted $C^*$-tensor product chain based on a single graded $C^*$-algebra, the sample algebra. The natural successive step is analysing the possibility of defining a natural action of the group of all finite permutations, denoted by $\bp$, on such an infinite twisted product. After constructing this action, the ultimate step is investigating the set of {\it symmetric states}, that is those invariant under such an action, and providing the analogous of the ``genuine'' De Finetti Theorem in this entirely new picture. 

One of the main ingredient in constructing the twisted tensor product is the {\it bicharacter}.
More precisely, the gradings on two $C^*$-algebras $\ga$ and $\gb$ is established by continuous actions $G\stackrel{\a}{\curvearrowright}\ga$ and 
$H\stackrel{\b}{\curvearrowright}\gb$ of two compact abelian groups $G,H$, so that a bicharacter $u$ is a mapping on the duals $u:\widehat{G}\times\widehat{H}\to\bt$ such that $u(\,\,,\t)$ is a character on $\widehat{G}$ for each fixed $\t\in\widehat{H}$ and $u(\s,\,\,)$ is a character on $\widehat{H}$ for each fixed $\s\in\widehat{G}$. In Quantum Physics, the involved bicharacter $u$ provides the commutation rules of fields, that is operators localised in separated regions. In constructing the infinite chain, we only need a single  $C^*$-dynamical system $(\gb,G,\b)$ and a bicharacter $u$ which must be hermitian, that is satisfying $u(\s,\t)=\overline{u(\t,\s)}$, $\s,\t\in\widehat{G}$.

As a first fact, we show that the construction of an infinite chain on which the group of all finite permutations is acting in a natural way (i.e. reproducing the flip of operators localised in two adjacent places) can be developed only when the involved bicharacter is hermitian.
Even if the investigation of symmetric states is certainly meaningful for such twisted systems based on hermitian bicharacters, we demonstrate that such an analysis on the structure of symmetric states and the corresponding De Finetti Theorem (cf. \cite{DeF, HS}) can be fruitfully carried out, using the standard techniques of Ergodic Theory, only for one more case involving the Klein four-group, and its Klein bicharacter, unique up to equivalence.

Exchangeable stochastic processes, connected with symmetric probability measures on the measurable space of events, are a relevant tool in classical probability theory. In this context, De Finetti Theorem deals with the relationship between independence and exchangeability.
More precisely, let $X_1, X_2,\dots$ be a countably infinite sequence of random variables. Such a sequence is said to be {\it exchangeable} if, for each integer $n$ and any choice of a finite subset $X_{i_1},\dots,X_{i_n}$ of $n$ distinct random variables, it has the same joint probability distribution of any other choice of $n$ distinct random variables $X_{j_1},\dots,X_{j_n}$. 
It is known that if the infinite sequence $X_1, X_2,\dots$ of random variables is independent and identically distributed, it is exchangeable, but the converse is false. De Finetti Theorem simply asserts that an exchangeable sequence is indeed a ``mixture" of independent and identically distributed sequences of random variables in the sense expressed first in \cite{DeF}, and then in the Hewitt-Savage formulation \cite{HS}.

Due to the application to an enormity of cases of interest, De Finetti Theorem represents a milestone in probability theory, and it has many interesting equivalent formulations. It should also be remarked that several attempts to generalise these De Finetti-like results to the quantum situation have been done in the last years. Concerning the standard formulation of De Finetti result in the quantum setting, we mention \cite{St2} for the usual tensor product (where the involved bicharacter is trivial), and \cite{Ffe} for the Fermi model (based on the {\it Fermi bicharacter}, encoding the Canonical Anti-commutation Relations). It is then meaningful to carry on the investigation of the structure of symmetric states on the new models of (infinite) twisted $C^*$-tensor products built exploiting the study in the previous papers \cite{FV, FV1}, then generalising the usual and Fermi models just mentioned.

After two sections (Sections \ref{prtel} and \ref{bptwtep}) collecting the basic facts concerning the $C^*$-twisted tensor product of two $C^*$-algebras w.r.t. the minimal norm, for the convenience of the reader we use some of such results, in particular Proposition \ref{nondeg}, reducing the matter to nondegenerate bicharacters,
to establish $*$-isomorphisms between noncommutative tori associated to a rational parameter $\th=m/n$ such that $\gcd(m,n)=1$. The most relevant result falling in this class is when $\th=1/2$, for which we recover the well-known Fermi tensor product between two copies of the algebra of all continuous functions $C(\bt)$ on the unit circle $\bt$. This is reported in Section \ref{ratogato}. Recall that the noncommutative tori, also known as rotation algebras, are among the most interesting examples of $C^*$-twisted tensor product, in both the rational and irrational cases.

Section \ref{visot} contains a surprising result, obtained by the 2nd named author, completely falling in group theory. Since the first step in constructing the action of the group of all finite permutations on the infinite chain is the flip on the twisted tensor product of two copies of the same graded algebra inherited by the action $G\stackrel{\b}{\curvearrowright}\gb$,
Proposition \ref{flip1} says that this is possible only when the involved nondegenerate bicharacter $u$ is hermitian: $u(\s,\t)=\overline{u(\t,\s)}$, for each $\s,\t\in\widehat{G}$. In such a case, the
{\it isotropy subgroup} is defined as $\D_{+,u}:=\{\sigma\in\widehat{G}\,;\, u(\sigma,\sigma)=1\}$. It is also shown (cf. Section \ref{ssttptep}) that a fruitful investigation of the symmetric states can be carried out using the standard notions of Ergodic Theory (following the lines of Section 3.1 of \cite{S}) if the restriction of $u$ to the product $\D_{+,u}\times\D_{+,u}$ of the isotropy subgroup is identically 1. It becomes then crucial to find discrete groups and bicharacters satisfying this condition. Here comes the result: if a bicharacter $u$ on $\widehat{G}$ satisfies the three conditions (cf. Theorem \ref{three})
\begin{itemize}
\item[{\bf-}\!{\bf-}] $u$ is hermitian,
\item[{\bf-}\!{\bf-}] $u$ is nondegenerate (cf. Section \ref{prtel}),
\item[{\bf-}\!{\bf-}] $u\lceil_{\D_{+,u}\times\D_{+,u}}=1$,
\end{itemize}
then only three possibilities turn out, up to isomorphisms: $G=\{1\}$ and $u$ is the trivial bicharacter on the sum of two copies of the trivial group $\widehat{G}=\{0\}$, producing the usual tensor product; $G=\bz_2$ and $u$ is the Fermi bicharacter on $\bz_2\oplus\bz_2$ yielding the Fermi model; $G=\bz_2\times \bz_2$ is the Klein four-group, and $u$ is the Klein bicharacter on two copies of $\bz_2\oplus\bz_2$ unique up to equivalence.

The successive construction of the infinite chain and the action of the finitary symmetric group on it are very delicate tools of technical nature. This is done in Sections \ref{asszac}, \ref{sevnif} and \ref{opytg}. Starting from an action $G\stackrel{\b}{\curvearrowright}\gb$ of the compact abelian group $G$ on the sample $C^*$-algebra $\gb$ and a bicharacter $u:\widehat{G}\times\widehat{G}\to\bt$, the infinite chain constructed by using the minimal $C^*$-norm is $\ga\equiv\gb\ot\gb\ot\cdots\ot\gb\ot\cdots$. If the bicharacter is hermitian, it is shown that the flips exchanging elements localised in two adjacent positions can provide an action of the group $\bp$ of all finite permutations (which is generated by all transpositions of two adjacent indices) on the chain $C^*$-algebra $\ga$.

Section \ref{free} is of explicative nature and contains motivations allowing the investigation of symmetric states associated to a 
probabilistic scheme in the general setting of Quantum Probability following the lines in \cite{CF2}.

Sections \ref{ssttptep} and \ref{klsce} contain the main ingredients to discuss the structure of the symmetric states. First, Theorem \ref{erg} describes the main properties of a symmetric state and, in particular, a necessary and sufficient condition under which the $*$-weakly compact convex set of symmetric states $\cs_\bp(\ga)$ is $\bp$-abelian (in the sense of \cite{S}, Definition 3.1.11). Under such a condition, we can deduce (cf. Remark \ref{choquet}) that $\cs_\bp(\ga)$ is a Choquet (indeed a Bauer) simplex and the quantum version of De Finetti Theorem holds true for such twisted models. At this stage, Corollary \ref{tpfekl} asserts that $\bp$-abelianess is granted when $u\lceil_{\D_{+,u}\times\D_{+,u}}=1$. Theorem \ref{three} then restricts the analysis to the three possibilities listed above.

Section \ref{klsce} is devoted to the detailed investigation of the third model relative to the Klein four-group, whereas
the concluding Section \ref{corem} lists some natural results, as well as open problems arising from our analysis. A part of this last section is devoted to the description of the types of von Neumann algebras $\pi_\om(\ga)''$ arising from the GNS representations of the symmetric states $\om$. A more detailed analysis relative to the extremal elements (i.e. infinite products of a single state on the sample algebra $\gb$) emerges when the action $\b$ providing the grading on $\gb$ is inner. This analysis takes advantage from the Klein transformation in Section 10 of \cite{FV}, recalled here in Section \ref{prtel}.

In Section \ref{corem}, we also discuss the possible investigation of symmetric states on other (compatible) $C^*$-completions, such as the max-norm one, as well as on the infinite noncommutative tori, where the involved bicharacters are not hermitian for $0<\theta<1/2$. As proved in \cite{ADR}, quite surprisingly, the class of symmetric states for such models is empty.

To end the present introduction, we point out the following considerations. First (cf. (i) in Theorem \ref{erg}), it emerges that any symmetric state is invariant under the, possibly proper, subgroup
$\D_{+,u}^\perp$, the annihilator of the isotropy subgroup $\D_{+,u}$. It obviously coincides with the whole acting group $G$ in the cases of usual tensor product and Fermi models. On the other hand, $\D_+^\perp$ is a proper subgroup of $G=\bz_2\times\bz_2$. Therefore, the Klein model allows a completely new investigation of its symmetric states. 

The Klein four-group is primarily known in elementary geometry as the symmetry group of a non-square rectangle. It
has also nice applications in music (it is the basic group of permutations in the {\it twelve-tone technique}). 
Concerning the purely mathematical aspects, we would like to mention the interesting paper \cite{Sta} in which, motivated by $K$-theory, an action of such a group is constructed on the irrational rotation algebra. Concerning our investigation, we note that 
the Klein group naturally arises in Section \ref{visot} as a result of abstract group theory with the scope to fruitfully study
the structure of symmetric states on an infinitely extended chain. It is then meaningful to ask about the possibility for the Klein group to be involved as a symmetry, like the Fermi case, for some relevant model in Quantum Field Theory. Here, we mention a collection of papers \cite{Mar, RW, Sch, SVJ, To} (quite far to be complete) in which this possibility is explored. However, at the best knowledge of the authors, there is no significant physical model having the Klein four-group as primary symmetry up to now. It is then desirable to push forward a systematic line of research on physical models for which such a group encodes the statistical properties of involved particles, as in the cases of Bose and Fermi situation.

\section{Preliminaries}
\label{prtel}

We gather here some useful information for the forthcoming sections. For the main properties and further details on the construction of the $C^*$-twisted tensor product, the reader is referred to \cite{FV, FV1}.

For each normed space $X$, the norm-closed linear subspace generated by any set $S\subset X$ is denoted as
$$
[S]:=\overline{{\rm span}(S)}^{\|\,\,\,\|_X}\,.
$$

\medskip

\noindent
\textbf{Finitary symmetric group.} 
Let $\ga$ be a unital $C^*$-algebra, and $G$ any discrete group which is acting on $\ga$ through an action 
$\bp\stackrel{\a}{\curvearrowright}\ga$ made of $*$-automorphisms of $\ga$.
In such a situation, the family of {\it invariant states} given by
$$
\cs_{G}(\ga):=\big\{\om\in\cs(\ga)\,;\,\om\circ\a_g=\om,\,\,\forall\,\,g\in G\big\}
$$ 
is a convex and $^*$-weakly compact subset of $\cs(\ga)$, with extremal points forming $\ce_\bp(\ga)\neq\varnothing$ (if $\ga\supsetneq\{0\}$), the family of the ($G$-)\textit{ergodic} states. 

Fix $\om\in\cs_{G}(\ga)$ and consider the GNS covariant representation $(\ch_\om,\pi_\om, V_\om,\xi_\om)$ associated to $\om$ (e.g. \cite{BR}). By von Neumann Mean Ergodic Theorem,
$$
\slim\limits_{n\to+\infty}\frac1{n}\sum_{k=0}^{n-1}V^k_\om=E_\om\,,
$$
where $E_\om$ is the selfadjoint projection onto the invariant vectors under $V_\om$.
With an abuse of notations, define the operator system 
$\pi_\omega(\ga)^G:=E_\omega\pi_\omega(\ga)E_\omega$. Observe that, for $x\in\ga$, 
$$
\pi_\omega(x)^G=0\iff\pi_\omega(x)(E_\om\ch_\omega)\subset(E_\om\ch_\omega)^\perp\,.
$$

We now specialise the situation to the group of all finite permutations of a given set.
Indeed, let $S$ be an infinite set and, for each finite set $F\subset S$, consider the symmetric group $\bp_F$ on $F$. The \textit{finitary symmetric group} is the direct limit of the class of $\bp_F$, for every finite set $F\subset S$:
\[
\bp_S:=\stackrel[\longrightarrow\,F]{}{\lim}\bp_F\,.
\]
$\bp_S$ is nothing but the group of permutations of $S$ moving only a finite number of its elements. 
For a $C^*$-dynamical system $(\ga,\bp_S,\a)$ consisting of a unital $C^*$-algebra $\ga$ and an action $\bp_S\stackrel{\a}{\curvearrowright}\ga$, the set of invariant states $\cs_{\bp_S}(\ga)$ is called the set of  \textit{symmetric states}.
In the present paper, the situation of interest is when $S=\bn$, the natural numbers, and simply denote $\bp:=\bp_\bn$.

In order to study the ergodic properties of symmetric states, we firstly recall the following combinatorial result which is interesting in itself. Set $\mathbf{n}:=\{0,1,\dots,n\}$. Notice that $\mathbf{n}$ consists of $n+1$ elements, and
$m\leq n$ $\Leftrightarrow$ $\mathbf{m}\subset\mathbf{n}$.
\begin{lem}[\cite{Ffe}, Lemma 5.1]\label{estimation}
Let $m,n\geq0$. For sufficiently large $N$, there exists some positive constant $C_{m,n}>0$ (depending on $m,n$ only) such that
\[
\frac{|\{g\in\bp_{\mathbf{N}}\colon\mathbf{m}\cap g(\mathbf{n})\neq\varnothing\}|}{N!}\leq C_{m,n}\,.
\]
\end{lem}

The sequence $\{g_n\}_{n\in\bn}\subset\bp$ defined, for $k\in \bn$, by
\[
g_n(k):=\begin{cases}
k+2^n\,&\text{if }0\leq k\leq 2^n-1\,,\\
k-2^{n}\,&\text{if }2^n\leq k\leq 2^{n+1}-1\,,\\
k\,&\text{if }2^{n+1}\leq k\,,
\end{cases}
\]
is useful in the sequel.

Finally, the {\it Cesàro mean} of any (vector or scalar) function $f$ will be denoted by, for the dummy variable $\hat g\in\bp$ and some appropriate limit (which is the usual one if $f$ is a numerical function or might be the limit in the strong or weak operator topology if $f$ is an operator-valued function),
$$
\cam\{f(\hat g)\}:=\lim_{N\to+\infty}\frac{1}{(N+1)!}\sum\limits_{g\in\bp_\mathbf{N}}f(g)\,,
$$
provided the above limit exists.

\medskip

\noindent
\textbf{Compact and discrete abelian groups.} 
In the present paper we mainly consider topological abelian groups, with the unique exception consisting of the permutation group described above. Such topological abelian groups are always supposed to be either Hausdorff and compact, or discrete. In the sequel, we omit to remark such algebraic/topological properties, just indicating when the group under consideration is compact or discrete. For $G$ in such class of groups, we indicate by
$\di\m_G=:\di g$ its Haar measure, normalised such that $\di\m_G(G)=1$ in the compact case and $\di\m_G(e_G)=1$ in the discrete one.

To focus our used notations, suppose that $G$ is compact. In such a situation, for the product operation we use the multiplicative convention and denote by 
$1\equiv e_G$ for its neutral element. 

The Pontryagin dual $\widehat{G}$ is discrete, and for such kind of groups we use the additive convention for the product operation denoting with 
$0\equiv e_{\widehat{G}}$ its neutral element.

\medskip

\noindent
\textbf{Bicharacters.} First of all, we report the definition of bicharacters on discrete abelian groups, see e.g. \cite{FV}. 
For $G,H$ discrete, a bicharacter is a function $u:\widehat{G}\times\widehat{H}\to\bt$ such that
\begin{align*}
&\widehat{G}\ni\s\mapsto u(\s,\t)\in\bt\,,\,\,\text{is a character for each fixed}\,\, \t\in\widehat{H}\,,\\
&\widehat{H}\ni\t\mapsto u(\s,\t)\in\bt\,,\,\,\text{is a character for each fixed}\,\, \s\in\widehat{G}\,.
\end{align*}

Let $A, B$ discrete abelian groups. If $u:A\times B\to\bt$ is a bicharacter, we define the left and right radical as
\begin{align*}
&A\supset{\rm Rad}_{\rm L}(u):=\{a\in A\colon u(a,b)=1,\,\,\forall b\in B\}\,,\\
&B\supset{\rm Rad}_{\rm R}(u):=\{b\in B\colon u(a,b)=1,\,\,\forall a\in A\}\,.
\end{align*}
It is immediate to show that ${\rm Rad}_{\rm L}(u)$ and ${\rm Rad}_{\rm R}(u)$ are both subgroups of $A$ and $B$, respectively. The bicharacter $u$ is said to be
{\it nondegenerate} if 
$$
{\rm Rad}_{\rm L}(u)=\{0_A\}\,\,\&\,\, {\rm Rad}_{\rm R}(u)=\{0_B\}\,.
$$

Suppose $A=B$, and let $u_1,u_2\colon A\times A\to\bt$ be two bicharacters. We say that $u_1$ and $u_2$ are \textit{equivalent}, and write $u_1\sim u_2$, if there exists a group automorphism $T\in\mathsf{Aut}(A)$ s.t. $u_1(\sigma,\tau)=u_2(T(\sigma),T(\tau))$ for every $\sigma,\tau\in A$.

\medskip

\noindent
\textbf{Twisted tensor products.} A $C^*$-dynamical system is a triple $(\ga,G,\a)$ describing a pointwise norm-continuous action $G\stackrel{\a}{\curvearrowright}\ga$
of the topological group $G$ on the $C^*$-algebra $\ga$. We consider the case when $\ga$ is unital (even if the non unital cases can be equally well investigated), and $G$ is a compact abelian group. In this situation, recall that the Pontryagin dual $\widehat{G}$ of $G$ is an abelian discrete group.

By performing the Fourier analysis, such an action $\a$ of $G$ induces a grading on $\ga$ 
through the linear projections 
$$
E^G_\s:=\int_G\overline{\chi_\s(g})\a_g(\,\cdot\,)\di g\,,\quad \s\in\widehat{G}\,.
$$
Since $\s_1\neq\s_2\Rightarrow E^G_{\s_1}(\ga)\cap E^G_{\s_2}(\ga)=\{0\}$, we can consider the (inner) direct sum, and we have that 
$\ga_o:=\dotplus_{\s\in\widehat{G}}\,E^G_\s(\ga)$ is dense in $\ga$. Elements $a\in E^G_\s(\ga)$ are called homogeneous of degree $\s$. Recall that $E_0^G$ consists of the usual conditional expectation of $\ga$ onto the invariant elements $E_0^G(\ga)\equiv\ga^G$ of $\ga$ under the action of $G$.
For a homogeneous elements $a\in E^G_\s(\ga)$, its {\it degree} $\partial a$ is precisely $\s\in\widehat{G}$.

Suppose we have two $C^*$-dynamical systems $(\ga,G,\a)$ and $(\gb,H,\b)$ with $G, H$ compact, together with a bicharacter
$u:\widehat{G}\times\widehat{H}\to\bt$. In such a situation, the linear tensor product $\ga_o\odot\gb_o$ is naturally equipped with
the involution and product given on homogeneous elements by
\begin{equation*}
(a\odot b)(A\odot B):=\overline{u(\partial A,\partial b)}aA\odot bB\quad
(a\odot b)^*:=\overline{u(\partial a,\partial b)}a^*\odot b^*\,.
\end{equation*}
It is uniquely extended on the whole $\ga_o\odot\gb_o$, making that an involutive algebra denoted by
$\ga_o\ot\gb_o$. The symbol $\ot$ is just used to point out that it is made up by using the bicharacter $u$.

Notice that, if $u_1$ and $u_2$ are two equivalent bicharacters on $\widehat{G}\times\widehat{H}$ (see above) then, necessarily, 
$\ga_o\oui\gb_o\sim\ga_o\ouii\gb_o$ as involutive algebras.

For the invariant states $\om\in\cs_G(\ga)$ and $\f\in\cs_H(\gb)$, we can consider the {\it product state} $\om\times\f$ on $\ga_o\ot\gb_o$ and the corresponding GNS representation
$\pi_{\om\times\f}$. It is shown that the class of product states separates the points of $\ga_o\ot\gb_o$: 
$$
(\om\times\f)(x^*x)=0\,\,\forall \om\in\cs_G(\ga)\,\,\text{and}\,\,\f\in\cs_H(\gb)\Rightarrow x=0\,.
$$

As for the simplest case when the bicharacter is trivial, and then the arising involutive algebra is just the algebraic tensor product algebra $\ga_o\otimes\gb_o$, the problem to investigating all $C^*$-norms on $\ga_o\ot\gb_o$ is a formidable task, see e.g. \cite{T}, Chapter IV. Here, for $x\in\ga_o\odot\gb_o$, we mention the minimal (or spatial) norm
$$
\|x\|_{\min}:=\sup\big\{\|\pi_{\om\times\f}(x)\|\,;\,\om\in\cs_G(\ga),\,\,\f\in\cs_H(\gb)\big\}\,,
$$
and the maximal (or universal) norm
$$
\|x\|_{\max}:=\sup\big\{\|\pi(x)\|\,;\,\pi\,\,\text{representation of}\,\,\ga_o\otimes\gb_o\big\} \,.
$$
The min and the max-seminorms are $C^*$-seminorms and, since $\|x\|_{\max}$ is finite and the product states separate the points,  they are indeed norms.
The corresponding completions of $\ga_o\ot\gb_o$ are denoted by $\ga\ot_{\min}\gb$ and $\ga\ot_{\max}\gb$, respectively.
 
 It is shown in \cite{FV1} that the max and the min-norm coincide if and only if either $\ga^G$ or $\gb^H$ is nuclear.

\section{Basic properties of the twisted tensor product}
\label{bptwtep}

Firstly, we complete Proposition 6.4. in \cite{FV}, making it suitable to our intents. Here, for a $C^*$-dynamical system $(\ga,G,\alpha)$, the involutive algebra generated by the homogeneous elements of $\ga$ will be denoted by $\ga^{(G)}_o$. Observe that if $N$ is a closed subgroup of $G$, then we get a $C^*$-dynamical system $(\ga,N,\alpha\lceil_{N})$, the homogeneous elements of which form $\ga_o^{(N)}$. It is clear that $\ga_o^{(N)}\supset\ga^{(G)}_o$, but we still write $\ga_o^{(N)}\cap\ga^{(G)}_o$ when viewing the elements of $\ga^{(G)}_o$ as graded by $\widehat{N}$.
\begin{prop}
\label{nondeg}
Let $(\ga,G,\alpha)$, $(\gb,H,\beta)$ be two $C^*$-dynamical systems. A bicharacter $u\colon\widehat{G}\times\widehat{H}\to\bt$ passes to the quotient modulo its left and right radicals $R$ and $S$, providing a nondegenerate bicharacter $w:(\widehat{G}/R)\times(\widehat{H}/S)\to\bt$. Then, $\ga^{(G)}_o\ot\gb^{(H)}_o\sim\big(\ga^{(R^\perp)}_o\cap\ga^{(G)}_o\big)\ow\big(\gb^{(S^\perp)}_o\cap\gb^{(H)}_o\big)\subset\ga^{(R^\perp)}_o\ow\gb^{(S^\perp)}_o$, where the latter twisted product is associated to the $C^*$-dynamical systems $\big(\ga,R^\perp,\alpha\lceil_{R^\perp}\big)$ and 
$\big(\gb,S^\perp,\beta\lceil_{S^\perp}\big)$. Moreover, $\ga\ot_{\min}\gb\sim\ga\ow_{\min}\gb$ as $C^*$-algebras.
\end{prop}
\begin{proof}
As mere linear spaces $\ga^{(G)}_o\odot\gb^{(H)}_o=\big(\ga^{(R^\perp)}_o\cap\ga^{(G)}_o\big)\odot\big(\gb^{(S^\perp)}_o\cap\gb^{(H)}_o\big)$, and it is straightforward to prove that they are isomorphic as involutive algebras:
\[
\ga^{(G)}_o\ot\gb^{(H)}_o\sim\big(\ga^{(R^\perp)}_o\cap\ga^{(G)}_o\big)\ow\big(\gb^{(S^\perp)}_o\cap\gb^{(H)}_o\big)\,.
\]
Now, since clearly $\cs_G(\ga)\subset\cs_{R^\perp}(\ga)$ and $\cs_H(\gb)\subset\cs_{S^\perp}(\gb)$, on $\ga_o^{(G)}\ot\gb_o^{(H)}$ we have
\[
\sup\limits_{\substack{\omega\in\cs_G(\ga)\\\varphi\in\cs_H(\gb)}}\|\pi_{\omega\times\varphi}(\cdot)\|\leq\sup\limits_{\substack{\rho\in\cs_{R^\perp}(\ga)\\\psi\in\cs_{S^\perp}(\gb)}}\|\pi_{\rho\times\psi}(\cdot)\|\,.
\]
On the other hand, $\sup\limits_{\substack{\omega\in\cs_G(\ga)\\\varphi\in\cs_H(\gb)}}\|\pi_{\omega\times\varphi}(\cdot)\|$ is a compatible $C^*$-norm for the twisted tensor product $\ga^{(R^\perp)}_o\ow\gb^{(S^\perp)}_o$ hence by minimality of its spatial norm among all the compatible ones (see Theorem 13.1 in \cite{FV}) on $\ga^{(R^\perp)}_o\ow\gb^{(S^\perp)}_o$ (\emph{a fortiori}, on $\big(\ga^{(R^\perp)}_o\cap\ga^{(G)}_o\big)\ow\big(\gb^{(S^\perp)}_o\cap\gb^{(H)}_o\big)$) we have
\[
\sup\limits_{\substack{\rho\in\cs_{R^\perp}(\ga)\\\psi\in\cs_{S^\perp}(\gb)}}\|\pi_{\rho\times\psi}(\cdot)\|\leq\sup\limits_{\substack{\omega\in\cs_G(\ga)\\\varphi\in\cs_H(\gb)}}\|\pi_{\omega\times\varphi}(\cdot)\|\,.
\]
It follows that the isomorphism of the involutive algebras $\ga^{(G)}_o\ot\gb^{(H)}_o$ and $\big(\ga^{(R^\perp)}_o\cap\ga^{(G)}_o\big)\ow\big(\gb^{(S^\perp)}_o\cap\gb^{(H)}_o\big)$ is min-min isometric and extends to their respective completions. 

To conclude, notice that $\big(\ga^{(R^\perp)}_o\cap\ga^{(G)}_o\big)\ow\big(\gb^{(S^\perp)}_o\cap\gb^{(H)}_o\big)$ is min-dense in $\ga^{(R^\perp)}_o\ow\gb^{(S^\perp)}_o$. Indeed, consider $a_i\in\ga^{(R^\perp)}_o,b_i\in\gb^{(S^\perp)}_o$ for $i=1,\dots,n$. Then, by density of $\ga_o^{(G)}$ in $\ga$ (and hence in $\ga^{(R^\perp)}_o$), as well as of $\gb_o^{(H)}$ in $\gb_o^{(S^\perp)}$, for every $\varepsilon>0$ there exist $a_{\varepsilon,i}\in\ga_o^{(G)},b_{\varepsilon,i}\in\gb_o^{(H)}$ such that $\|a_i-a_{\varepsilon,i}\|_\ga,\|b_i-b_{\varepsilon,i}\|_\gb<\varepsilon$. Since the spatial norm is sub-cross,
\begin{multline*}
\left\|\sum\limits_{i=1}^na_i\ow b_i-\sum\limits_{i=1}^na_{\varepsilon,i}\ow b_{\varepsilon,i}\right\|_{\min}\leq\sum\limits_{i=1}^n\bigg(\|a_i-a_{\varepsilon,i}\|_\ga\|b_i\|_\gb+\|a_{\varepsilon,i}\|_\ga\|b_i-b_{\varepsilon,i}\|\bigg)\\<\varepsilon\left(n\varepsilon+\sum\limits_{i=1}^n(\|a_i\|_\ga+\|b_i\|_\gb)\right)\,.
\end{multline*}
It follows that the min-completion of $\big(\ga^{(R^\perp)}_o\cap\ga^{(G)}_o\big)\ow\big(\gb^{(S^\perp)}_o\cap\gb^{(H)}_o\big)$ is exactly $\ga\ow_{\min}\gb$, so that $\ga\ot_{\min}\gb\sim\ga\ow_{\min}\gb$ as $C^*$-algebras.
\end{proof}
Analogously, we provide the following complete version of Proposition 7.1. in \cite{FV}.
\begin{prop}
\label{flip1}
Once set $v(\t,\s):=\overline{u(\s,\t)}$ for every $\sigma,\tau\in\widehat{G}$, the flip $\F_o:\gb_o\ov\ga_o\to\ga_o\ot\gb_o$ defined in Proposition 7.1 of \cite{FV} by
$$
\F_o(b\odot a):=\overline{u(a,b)}a\odot b\,,\quad a\in\ga\,\,\text{and}\,\,b\in\gb\,\,\text{homogeneous}\,,
$$
extends to a $*$-isomorphism
$\F:\gb\ov_{\min}\ga\to\ga\ot_{\min}\gb$ intertwining the action of $H\times G$ with that of $G\times H$:
$$
\F\circ(\b_h\ov_{\min}\a_g)=(\a_g\ot_{\min}\b_h)\circ\F\,.
$$
\end{prop}
\begin{proof}
That $\F_o$ is a $*$-algebra homomorphism is easily verifiable by a straightforward computation. Let us check that it is isometric w.r.t. the corresponding minimal $C^*$-norms. 
For such a purpose, we start by performing a computation for homogeneous elements $a\in\ga_o$, $b\in\gb_o$, and invariant states $\om\in\cs_G(\ga)$ and $\f\in\cs_H(\gb)$.

Indeed, if $V\colon\ch_\f\otimes\ch_\om\to\ch_\om\otimes\ch_\f$ is the swapping unitary operator
\begin{align*}
({\pi_\om}_{U_\om}\ot\pi_\f)&(\F_o(b\odot a))=\overline{u(a,b)}\pi_\om(a)U_\om(g_{\partial b})\otimes\pi_\f(b)\\
=&\overline{u(a,b)}V\big(\pi_\f(b)\otimes U_\om(g_{\partial b})\big(U_\om(g_{\partial b})^*{\pi_\om}(a)U_\om(g_{\partial b}\big)\big)V^*\\
=&\overline{u(a,b)}V\big(\pi_\f(b)\otimes U_\om(g_{\partial b})\pi_\om(\a_{g^{-1}_{\partial b}}(a))\big)V^*\\
=&V\big(\pi_\f(b)\otimes U_\om(g_{\partial b})\pi_\om(a)\big)V^*\\
=&V\big(\pi_\f(b)\otimes U_\om({}_{\partial b}g)\pi_\om(a)\big)V^*\\
=&V(\pi_\f\ov {}_{U_\om}\pi_\om)(b\odot a)V^*\,.
\end{align*}
where we have used the fact that $g_{\partial b}={}_{\partial b}g$, since
\begin{equation}
\label{lgtr}
\chi_{g_{\partial b}}(\sigma)=\overline{u(\sigma,\partial b)}=v(\partial b,\sigma)=\chi_{{}_{\partial b}g}(\sigma),\,\,\sigma\in\widehat{G}\,.
\end{equation}

We have then proved that 
$$
\|\big(({\pi_\om}_{U_\om}\ot\pi_\f)\circ\F_o\big)(x)\|_{\cb(\ch_\omega\otimes\ch_\varphi)}=\|(\pi_\f\ov {}_{U_\om}\pi_\om)(x)\|_{\cb(\ch_\varphi\otimes\ch_\omega)}\,,\quad x\in\gb_o\ov\ga_o\,,
$$
and thus $\F_o$ extends to a $*$-isomorphism $\F:\gb\ov_{\min}\ga\to\ga\ot_{\min}\gb$ between the completions w.r.t. the corresponding minimal $C^*$-norms.

Lastly, the intertwining property of $\F$ follows from an easy check on the algebraic parts and the compatibility of the minimal norm.
\end{proof}
We end the present section briefly outlining the combined effects of the Klein transformations realising the isomorphism between twisted and usual (minimal) tensor products with the flips, provided the former exists. We refer the reader to \cite{FV} for details and proofs.

Indeed, suppose that the action $\a$ of the group $G$ on the first algebra $\ga$ is inner:
\begin{itemize}
\item[(i)] there is a unitary representation $G\ni g\mapsto u(g)\in\cu(\ga)$ such that
$$
\a_g(a)=u(g)au(g^{-1})\,,\quad a\in\ga,\,\,g\in G\,;
$$
\item[(ii)] such a representation $G\ni g\mapsto u(g)\in\cu(\ga)$ is continuous when $\ga$ is equipped with the seminorms 
$$
p_{\pi,\xi}(a):=\|\pi(a)\xi\|\,,\quad \pi\in\re(\ga),\,\,\xi\in\ch_\pi\,.
$$
\end{itemize}
By using \eqref{lgtr} to define
$$
g_\t\equiv{}_\t g\in G\,,\quad \t\in\widehat{H}\,,
$$
the left and right Klein transformations $\k_{\rm L}:\ga\ot_{\min}\gb\to\ga\otimes_{\min}\gb$ and $\k_{\rm R}:\gb\ov_{\min}\ga\to\gb\otimes_{\min}\ga$, are given on generators by
\begin{align*}
&\ga_o\odot E_\t^H(\gb)\ni a\ot b\mapsto\k_{\rm L}(a\ot b):=(a u_{\t})\otimes b\in\ga_o\otimes\gb_o\,,\\
&E_\t^H(\gb)\odot\ga_o\ni b\ot a\mapsto\k_{\rm R}(b\ot a):=b\otimes u_{\t}a\in\gb_o\otimes\ga_o\,.
\end{align*}
Such Klein transformations are indeed $*$-isomorphisms.

The left Klein transformation intertwines the product actions:
$$\k_{\rm L}\circ(\a\ot_{\min}\b)=(\a\otimes_{\min}\b)\circ\k_{\rm L}\,
$$
and sends product states onto invariant product states:
$$
\k_{\rm L}^{\rm t}(\psi_{\om,\f})=\om\times\f\,,\quad \om\in\cs(\ga),\,\,\f\in\cs_H(\gb)\,.
$$
The right Klein transformation satisfies the analogous properties:
\begin{align*}
&\k_{\rm R}\circ(\b\ov_{\min}\a)=\b\otimes_{\min}\a\,,\\
&\k_{\rm L}^{\rm t}(\psi_{\f,\om})=\f\times\om\,,\quad \om\in\cs(\ga),\,\,\f\in\cs_H(\gb)\,.
\end{align*}

We now notice that the flip $\F$ in Proposition \ref{flip1} realises a $*$-isomorphism between $\gb\ov_{\min}\ga$ and $\ga\ot_{\min}\gb$. After denoting by 
$\s:\gb\otimes_{\min}\ga\to\ga\otimes_{\min}\gb$ the flip realising the $*$-isomorphism between the corresponding usual minimal tensor products, the combined effects of the flips and the Klein transformations and their transposed maps can be summarised in commutative diagrams. We leave the details to the reader.

\section{Rational noncommutative tori}
\label{ratogato}

The \emph{rational} rotation $C^*$-algebra gives a pedagogical example of the construction in Proposition \ref{nondeg}. Even if the irrational rotation algebra is the most studied example, the rational rotation algebra yet presents many interesting aspects, see e.g. \cite{Bo}. Denote with $\ba_\th$ the noncommutative torus associated to $\th\in\br$.

Indeed, let $G=\bt=H$ the unit circle, acting on $\ga=C(\bt)=\gb$ via the usual rotation $\alpha_z(f)=f(z\,\cdot)=\beta_z(f)$ for each $z\in\bt$ and $f\in C(\bt)$. Let $u(x,y):=e^{i2\pi\frac{m}{n}xy}$ ($x,y\in\bz$) with $m,n\in\bn$ coprime. Using the notation in Proposition \ref{nondeg},
\[
R=S=\text{Rad}(u)=n\bz\,,
\]
so that
\begin{itemize}
\item $L^\perp=R^\perp\sim\bz_n\subset\bt$ and $\widehat{L}=\widehat{R}\sim\bt/\bz_n$;
\item $v\colon\bz_n\times\bz_n\to\bt$, $v(x,y)=e^{i\frac{2\pi}{n}mxy}$, $x,y\in\bz_n$;
\item $w\colon\bz_n\times\bz_n\to\bt$, $w(x,y)=e^{i\frac{2\pi}{n}(n-m)xy}$, $x,y\in\bz_n$;
\item $C(\bt)^\bt_o=\underset{k\in\bz}{\dotplus}\bc z^k$;
\item $C(\bt)^{\bz_n}_o\cap C(\bt)^\bt_o=\bigoplus\limits_{j=0}^{n-1}\overline{\underset{k\in\bz}{\dotplus}\bc z^{j+kn}}$;
\end{itemize}
and, in view of the last part of Proposition \ref{nondeg},
\[
\ba_{m/n}\equiv C(\bt)\ot_{\min}C(\bt)\sim C(\bt)\ov_{\min} C(\bt)
\]
as $C^*$-algebras. We note that, by nuclearity of $C(\bt)$, the minimal completion is the unique $C^*$-completion w.r.t. all compatible $C^*$-norms, see \cite{FV1}, Thm. 6.4. 
At this stage, we would like to note that
\begin{itemize}
\item the previous analysis can be carried out also for general (unital, to simplify) $C^*$-algebra $\ga$ equipped with an action of $\bt$ instead of the particular case $C(\bt)$: 
$\ga\ot_{\min}\ga\sim \ga\ov_{\min} \ga$;
\item by taking into account the well-known $*$-isomorphism $\ba_\th\sim\ba_{1-\th}$ for $\th\in(0,1/2)$, we again have
$$
C(\bt)\ov C(\bt)\sim\ba_{m/n}\sim\ba_{1-m/n}\sim C(\bt)\ow C(\bt)\,.
$$
\end{itemize}

A case of interest is that when $m=1$ and $n=2$ which corresponds to the Fermi twisted tensor product, see e.g. \cite{Ffe}. Indeed, for the involved bicharacters, $v=w$ both coinciding with the Fermi bicharacter given in (2) of Theorem \ref{three}, and then $\ba_{1/2}\sim C(\bt)\of_{\rm min}C(\bt)$.

\section{Hermitian bicharacters on discrete abelian groups}
\label{visot}

In order to built multiple twisted tensor products based to a single $C^*$-dynamical system $(\ga, G,\a)$ as above,
on which acts in a natural way the permutations exchanging the indices, we easily deduce a necessary (and also sufficient) constrain on the involved bicharacter. It automatically follows by \cite{FV}, Proposition 7.1 and select only twisted tensor products associated to hermitian bicharacter on $\widehat{G}\times\widehat{G}$. 

A bicharacter $u:A\times A\to\bt$ on the discrete abelian group $A$ is {\it hermitian} if, by definition,
$$
u(a,b)=\overline{u(b,a)}\,,\quad a,b\in A\,.
$$

As suggested by Prop. \ref{nondeg}, we will see below that, in constructing (infinite) $C^*$-twisted tensor product, we can reduce the matter to nondegenerate bicharacters. Therefore, we directly deal with hermitian nondegenerate bicharacters.
We also note that, as suggested in \cite{St2} (see also \cite{Ffe}), in order to study the structure of symmetric states on the $C^*$-twisted tensor product chain by using the standard results of Ergodic Theory, we need a further condition on the common discrete group on (the product of two copies of) which is defined the bicharacter, that we are going to explain.

For the hermitian nondegenerate bicharacter $u:A\times A\to\bt$ on the discrete group $A$, we define the {\it isotropy subgroup} $\D_+\subset A$ given by
$$
\D_+\equiv\Delta_{+,u}:=\{a\in A\,;\, u(a,a)=1\}\,.
$$
It is elementary to see that $\D_+$ is indeed a subgroup of $A$. We obtain the following surprising result falling in the abstract theory of discrete abelian groups which has then a self-containing interest.
\begin{thm}[Vincenzi, \cite{V}]
\label{three}
Let $u:\widehat{G}\times\widehat{G}\to\bt$ be a hermitian nondegenerate bicharacter on the dual of the compact abelian group $G$, with $\D_+$ the corresponding isotropy subgroup. Suppose further that
$u\lceil_{\D_+\times\D_+}=1$. Then we have only the following three possibilities, up to the equivalence described in Section \ref{prtel}:
\begin{itemize}
\item[(1)] the trivial situation $G=\{1\}$ and $\widehat{G}=\{0\}$, with $u$ being the trivial bicharacter;
\item[(2)] $G$ the multiplicative group $\bz_2=\{1,-1\}$ and $\widehat{G}$ the additive group $\bz_2=\{0,1\}${\rm-mod 1}, with $u:=u_{\rm F}$ being the Fermi bicharacter
$$
u_{\rm F}(m,n)=(-1)^{mn}\,,\quad m,n\in \{0,1\}\,;
$$
\item[(2)] $G$ the multiplicative Klein group $\bz_2\times\bz_2$ and $\widehat{G}$ the additive Klein group $\bz_2\oplus\bz_2$, with $u:=u_{\rm K}$ being the Klein bicharacter
$$
u_{\rm K}({\bf m},{\bf n})=(-1)^{m_1n_1+m_2n_2}\,,\quad m_i,n_i\in \{0,1\},\,\, i=1,2\,.
$$
\end{itemize}
\end{thm}
\begin{proof}
Since $\Delta_+=\ker\phi$, where $\phi\in\text{Hom}(\widehat{G},\bz_2)$ is defined by $\phi(\sigma)=u(\sigma,\sigma)$, then
$\big|\widehat{G}\colon\Delta_+\big|=\big|\widehat{G}\colon\ker\phi\big|=|\text{ran }\phi|\leq|\bz_2|=2$. If $\big|\widehat{G}\colon\Delta_+\big|=1$ (i.e. $\Delta_+=\widehat{G}$), then $u\equiv1$, thus it is nondegenerate if and only if $\widehat{G}=\{0\}$. Therefore, from now on, we suppose that $\big|\widehat{G}\colon\Delta_+\big|=2$. The inclusion map $\iota\colon\Delta_+\hookrightarrow \widehat{G}$ induces a canonical surjection $\pi\colon \widehat{\widehat{G}}\twoheadrightarrow\widehat{\Delta_+}$ such that 
$$
|\ker\pi|=\Big|\widehat{\widehat{\phantom{G}}\!\!\!\!\!G/\Delta_+}\Big|=|\bz_2|=2\,.
$$
On the other hand, since $u$ is nondegenerate, the map
\[
\begin{aligned}
\gamma\colon \widehat{G}&\hookrightarrow\widehat{\widehat{G}}\\
\sigma&\mapsto u(\sigma,\cdot)
\end{aligned}
\]
is a group monomorphism. Now, $u\lceil_{\Delta_+\times\Delta_+}\equiv1$, and thus the composition
\[
\Delta_+\xhookrightarrow{\gamma\circ\iota}\widehat{\widehat{G}}\stackrel{\pi}{\twoheadrightarrow}\widehat{\Delta_+}
\]
is the trivial homomorphism. This means that $(\gamma\circ\iota)(\Delta_+)\subset\ker\pi$, hence
\[
|\Delta_+|=|(\gamma\circ\iota)(\Delta_+)|\leq|\ker\pi|=2\,.
\]
Since $\big|\widehat{G}\colon\Delta_+\big|=2$ by assumption, by Lagrange's theorem
\[
|\widehat{G}|=\big|\widehat{G}\colon\Delta_+\big||\Delta_+|=2|\Delta_+|
\]
whence $|\widehat{G}|=2$ or $4$. If $|\widehat{G}|=2$, then $\widehat{G}\sim\bz_2$, $u\sim u_{\rm{F}}$ and $\Delta_+=\{0\}$. If instead $|\widehat{G}|=4$, then either $\widehat{G}\sim\bz_4$ or $\widehat{G}\sim K_4:=\bz_2\oplus\bz_2$. On the one hand, the only two nondegenerate bicharacters on $\bz_4$ have the form
\[
u_\varepsilon(m,n)=\imath^{\,\varepsilon\,m\,n}\quad(m,n\in\bz_4)
\]
for $\varepsilon\in\{\pm1\}$, thus they are not hermitian. On the other hand, $K_4$ admits a unique (up to equivalence) nondegenerate hermitian bicharacter such that $|K_4\colon\Delta_+|=2$, that is 
$$
u_{\rm K}({\bf m},{\bf n})=(-1)^{m_1n_1+m_2n_2}\,,\quad m_j,n_j\in \{0,1\},\,\, j=1,2\,.
$$
Observe that $\Delta_+=\{(0,0),(1,1)\}$ and $u_{\rm{K}}\lceil_{\Delta_+\times\Delta_+}\equiv1$. Then, $u\sim u_{\rm K}$.
\end{proof}
Two of these cases provide already investigated situations. In fact, (1) corresponds to the usual tensor product studied in \cite{St2} and (2) corresponds to the Fermi tensor product studied in \cite{Ffe}, whereas the case (3) is completely new.

\section{Associativity of the minimal twisted tensor product}
\label{asszac}

Every pair of $C^*$-dynamical systems $(\gb,G,\beta)$, $(\gc,G,\gamma)$, together with a bicharacter $u$ on 
$\widehat{G}\times\widehat{G}$, induces a new system $(\gb\ot\gc,G,\delta^{(\beta,\gamma)})$ where $\delta^{(\beta,\gamma)}$ is the diagonal action of $G$ on $\gb\ot\gc$:
\[
\delta^{(\beta,\gamma)}_g(b\ot c):=\beta_g(b)\ot\gamma_g(c),\quad g\in G,b\in\gb,c\in\gc\,.
\]
Then, surely $\gb_o\ot\gc_o\subseteq(\gb\ot\gc)_o$, but the inclusion may well be proper. For instance, if
$
(\gb,G,\beta)=(C(\bt),\bz_2,\th)=(\gc,G,\gamma)
$
where $\th$ is the involution induced by the $\pi$-rotation on $\bt$, then for any $A,B\in\bc$ and $f,g\in C(\bt)$,
\[
A\cos f\ot\cos g+B\sin f\ot\sin g
\in\left(C(\bt)\ot_{\min}C(\bt)\right)_o\smallsetminus\left(C(\bt)_o\ot C(\bt)_o\right)
\]
where
\[
\begin{cases}
\cos(f):=\ulim\limits_{n\to\infty}\sum\limits_{0\leq k\leq n}(-1)^k\frac{f^{2k}}{(2k)!}\\
\sin(f):=\ulim\limits_{n\to\infty}\sum\limits_{0\leq k\leq n}(-1)^k\frac{f^{2k+1}}{(2k+1)!}
\end{cases}
\]
Nonetheless, $G$ still acts diagonally on $\gb_o\ot\gc_o$, so that $(\gb_o\ot\gc_o,\|\,\,\|_{\min})$ is well a $\widehat{G}$-graded pre-$C^*$-algebra, obviously dense in $\gb\ot\gc$. As a result, it makes sense to study $(\gb_o\ot\gc_o)\ot\gd_o$ for any other $C^*$-dynamical system $(\gd,G,\lambda)$. The following lemma asserts that the minimal twisted tensor product is associative in this particular case, thus giving a non-ambiguous meaning to expressions like $\ga_2:=\gb\ot\gb\ot\gb$, and more in general to the minimal twisted tensor product $\ga_n$ of $n+1$ copies of $\gb$.
\begin{lem}\label{associativity}
With the notation above, $(\gb_o\ot\gc_o)\ot\gd_o\sim\gb_o\ot(\gc_o\ot\gd_o)$ as involutive algebras. The isomorphism extends to the minimal completions so that $(\gb\ot\gc)\ot\gd\sim\gb\ot(\gc\ot\gd)$.
\end{lem}
\begin{proof}
By associativity of the tensor product $\odot$, it is clear that $(\gb_o\ot\gc_o)\ot\gd_o\sim\gb_o\ot(\gc_o\ot\gd_o)$ as linear spaces. Let $*_{L}$ and ${*}_R$ be the adjoint operations on the involutive algebras $(\gb_o\ot\gc_o)\ot\gd_o$ and $\gb_o\ot(\gc_o\ot\gd_o)$, respectively. Similarly, let $\cdot_{L}$ and $\cdot_R$ be their respective products. For homogeneous $x_1\in\gb_o$, $x_2\in\gc_o$ and $x_3\in\gd_o$ of degree $\sigma_i\in\widehat{G}$ ($i=1,2,3$) respectively, by recalling that $\partial(x_i\ot x_j)=\sigma_i+\sigma_j\in\widehat{G}$, we get
\begin{multline*}
((x_1\ot x_2)\ot x_3)^{*L}=\overline{u(\sigma_1+\sigma_2,\sigma_3)}(x_1\ot x_2)^*\ot x_3^*=\\=\overline{u(\sigma_1+\sigma_2,\sigma_3)}\,\,\overline{u(\sigma_1,\sigma_2)}(x_1^*\ot x_2^*)\ot x_3^*=\\=\overline{u(\sigma_1,\sigma_2+\sigma_3)}\,\,\overline{u(\sigma_2,\sigma_3)}x_1^*\ot (x_2^*\ot x_3^*)=\\=\overline{u(\sigma_1,\sigma_2+\sigma_3)}x_1^*\ot (x_2\ot x_3)^*=(x_1\ot (x_2\ot x_3))^{*R}
\end{multline*}
where we have used the associativity of $\odot$ on the elementary tensors. If $X_1\in\gb_o$, $X_2\in\gc_o$ and $X_3\in\gd_o$ are other three homogeneous elements of degree $\tau_i\in\widehat{G}$ ($i=1,2,3$), again using the associativity of $\odot$,
\begin{multline*}
((x_1\ot x_2)\ot x_3)\cdot_L((X_1\ot X_2)\ot X_3)=\\=\overline{u(\tau_1+\tau_2,\sigma_3)}\left((x_1\ot x_2)(X_1\ot X_2)\right)\ot x_3X_3=\\
=\overline{u(\tau_1+\tau_2,\sigma_3)}\,\,\overline{u(\tau_1,\sigma_2)}(x_1X_1\ot x_2X_2)\ot x_3X_3=\\=\overline{u(\tau_1,\sigma_2+\sigma_3)}\,\,\overline{u(\tau_2,\sigma_3)}x_1X_1\ot(x_2X_2\ot x_3X_3)=\\=\overline{u(\tau_1,\sigma_2+\sigma_3)}x_1X_1\ot\left((x_2\ot x_3)(X_2\ot X_3)\right)=\\=(x_1\ot(x_2\ot x_3))\cdot_R(X_1\ot(X_2\ot X_3))\,.
\end{multline*}
By extending the equality to all the finite linear combinations of tensor products of homogeneous elements, it straightforwardly results that
\[
(\gb_o\ot\gc_o)\ot\gd_o\sim\gb_o\ot(\gc_o\ot\gd_o)
\]
as involutive algebras, and hence $\overline{(\gb_o\ot\gc_o)\ot\gd_o}^{\min}\sim\overline{\gb_o\ot(\gc_o\ot\gd_o)}^{\min}$. Since the min-norm is cross (Proposition 12.7 in \cite{FV}), by Proposition 12.5 in \cite{FV} $\overline{(\gb_o\ot\gc_o)\ot\gd_o}^{\min}$ contains isomorphic copies of the completion of its marginals pre-$C^*$-algebras i.e. \[
\gb\ot\gc,\gd\hookrightarrow\overline{(\gb_o\ot\gc_o)\ot\gd_o}^{\min}
\]
Analogously, $\gb,\gc\ot\gd\hookrightarrow\overline{\gb_o\ot(\gc_o\ot\gd_o)}^{\min}$. In conclusion,
\[
(\gb\ot\gc)\ot\gd\sim\gb\ot(\gc\ot\gd)\,.\qedhere
\]
\end{proof}
Lemma \ref{associativity} allows us to write $(\gb\ot\gc)\ot\gd\sim\gb\ot(\gc\ot\gd)$ simply as $\gb\ot\gc\ot\gd$ with no problems of ambiguity. In particular, we will be able to construct the infinite twisted $C^*$-tensor product without concerning for where to put brackets. This allows us to provide a natural action of the finitary symmetric group $\bp$ associated to the exchange of the indices in the chain, without any trouble associated to parenthesisation, see the next two sections.

\section{The infinite twisted $C^*$-tensor product}
\label{sevnif}

The present section is devoted to the construction of the $C^*$-inductive limit of a system of twisted minimal $C^*$-tensor products based on a fixed $C^*$-dynamical system $(\gb,G,\b)$, where $G$ is a compact abelian group acting on the unital $C^*$-algebra $\gb$. The first step of our iterative construction will be the $C^*$-dynamical system $(\gb\ot\gb,G,\delta^{(\beta)})$ where $\gb\ot\gb$ is the completion of the twisted tensor product $\gb_o\ot\gb_o$ w.r.t the min-norm and 
$$
\delta^{(\beta)}_g(x):=(\b_g\ot\b_g)(x)\,,\quad x\in\gb\ot\gb,\,\,g\in G,
$$
is the diagonal action of $G$ on $\gb\ot\gb$. Dealing with minimal $C^*$-tensor completions only, from now on we will always omit all script ``min".

We start by defining a direct system of $C^*$-algebras based on the single $C^*$-dynamical system $(\gb,G,\b)$ as above by using the $*$-monomorphism
\[
\iota_\ga:\ga\hookrightarrow\ga\ot\gb
\]
of Proposition 12.5 in \cite{FV}. Indeed, define
\begin{align*}
&\ga_{0}:=\gb\,,\quad \ga_{n+1}:=\ga_{n}\ot\gb\,,\\
&\iota_n:\ga_{n}\hookrightarrow\ga_{n+1}\,,\quad n\in\bn\,.
\end{align*}
In addition, we put
$$
\a^{(0)}_g:=\b_g\in\mathsf{Aut}(\ga_0), \,\,\a^{(n+1)}_g:=\a^{(n)}_g\ot\b_g\in\mathsf{Aut}(\ga_{n+1}),\quad g\in G,\,\,n\in\bn
$$
and, for a sequence $(\f_n)_n\subset\cs_G(\gb)$
$$
\om_0:=\f_0,\,\,\om_{n+1}:=\om_n\times\f_{n+1}\in,\,\,n\in\bn.
$$
Observe that, for every $n\in\bn$, $\om_{n}$ belongs to $\cs(\ga_n)$ (and in particular to $\cs_G(\ga_n)$) thanks to Proposition 9.1 in \cite{FV}. Furthermore, the automorphisms $\left(\alpha_g^{(n)}\right)_{\substack{n\in\bn\\g\in G}}$ and the $G$-invariant states $\left(\omega_n\right)_{n\in\bn}$ evidently satisfy the compatibility relations
\begin{equation}\label{rel}
\a_g^{(n+1)}\circ\iota_n=\iota_{n}\circ\a_g^{(n)},\,\,\om_{n+1}\circ\iota_n=\om_n,\quad n\in\bn,\,\,g\in G\,.
\end{equation}
Setting $\F_{ij}:=\iota_{i-1}\circ\cdots\circ\iota_{j+1}\circ\iota_j$, it is immediate to verify that $(\ga_j, \F_{ij})_{\{i,j\in\bn\,;\,i>j\}}$ is a direct system of $C^*$-algebras, thus yielding the {\it direct limit}
$$
\ga_\infty:=\stackrel[\longrightarrow\,n]{}{\lim}\ga_n\,.
$$
We also notice that $\alpha^{(n)}\colon g\mapsto\alpha^{(n)}_g$ defines a pointwise norm-continuous action of $G$ on $\ga_n$ for each $n\in\bn$. By (\ref{rel}), the sequence of actions $\Big(G\stackrel{\a^{(n)}}{\curvearrowright}\ga_n\Big)_n$, and the product states $(\om_n)_n$ are compatible with the above direct system, thus providing an action $\big(G\stackrel{\a^{(\infty)}}{\curvearrowright}\ga_\infty\big)_n$ and a positive unital functional $\om_\infty$ on $\ga_\infty$, see e.g. Appendix L1 in \cite{WO}. Therefore, by Appendix L2 in \cite{WO}, we can firstly define the $C^*$-inductive limit $\ga:=\overline{\ga_\infty}$ associated to the above direct system $(\ga_j, \F_{ij})_{\{i,j\in\bn\,;\,i>j\}}$ as the completion of its direct limit $\ga_\infty$, and secondly an action $G\stackrel{\a}{\curvearrowright}\ga$ and an infinite product state $\om=\f_0\times\f_1\cdots\in\cs(\ga)$ which is also invariant for the inductive limit action $\a$ of $G$ on $\ga$, i.e. $\omega\in\cs_G(\ga)$. We collect these results in the following
\begin{thm}
The inductive sequence $(\ga_n)_n$ of $C^*$-algebras, together with the compatible sequence of actions $\Big(G\stackrel{\a^{(n)}}{\curvearrowright}\ga_n\Big)_n$, determines a $C^*$-dynamical system $(\ga,G,\a)$, referred to as the (minimal) $C^*$-inductive limit of a countably infinite number of copies of a single $C^*$-dynamical system $(\gb,G,\b)$. In addition, any sequence of invariant states $(\f_n)_n\subset\cs_G(\gb)$ determines an invariant state $\om\in\cs_G(\ga)$, the infinite product state, through its compatible sequence of partial product states $(\om_n)_n$.
\end{thm}
It is not hard to see that to obtain a product state on the whole inductive limit $\ga$, it is enough to start with a sequence $(\f_n)_n\subset\cs(\gb)$ of states which are all invariant but at most one, see \cite{AM1}.

For any $n\in\bn$ and $C_n:=\frac{1}{n+1}\binom{2n}{n}$ the Catalan numbers, there are exactly $C_n$ ways of inserting brackets in a chain of $n+1$ copies of $\gb$ in order to associate the products $\ot$ among them. For example, for $n=3$, there are $C_3=5$ ways to parenthesise a chain of $4$ copies of $\gb$:
\begin{multline*}
(\gb\ot\gb)\ot(\gb\ot\gb),\,((\gb\ot\gb)\ot\gb)\ot\gb=\ga_3,\\(\gb\ot(\gb\ot\gb))\ot\gb,\,\gb\ot((\gb\ot\gb)\ot\gb),\,\gb\ot(\gb\ot(\gb\ot\gb))\,.
\end{multline*}
In view of Lemma \ref{associativity}, all the five $C^*$-algebras above are isomorphic. More in general, for every $n\in\bn$, all the $C_n$ possible parenthesisations are isomorphic $C^*$-algebras. The previous considerations allows us to simply write $\ga_n=\underbrace{\gb\ot\dots\ot\gb}_{(n+1)\text{-copies}}$. Analogously, $\ga_\infty=\stackrel[\longrightarrow\,n]{}{\lim}\underbrace{\gb\ot\dots\ot\gb}_{(n+1)\text{-copies}}$. It is then meaningful to write $\underset{n\in\bn}{\ot}\gb:=\overline{\ga_\infty}=\ga$, exactly as for the usual tensor product and the Fermi one, see e.g. \cite{St2, Ffe}.

\section{The action of the finitary symmetric group on the infinite twisted tensor product}
\label{opytg}

From now on, we will deal with nondegenerate \emph{hermitian} bicharacters $u\in\mathsf{Bic}(\widehat{G})$ only, i.e.
$$
\widehat{G}\times\widehat{G}\ni(\s,\t)\mapsto u(\s,\t)=\overline{u(\t,\s)}\in\bt\,.
$$

The reason of this assumption is that in such a case the flip $\F$ of Proposition \ref{flip1} realises an involutive $*$-automorphism of $\gb\ot\gb$ which commutes with the product action of $G\times G$, and \emph{a fortiori} with the diagonal action of $G$. Analogously, for every discrete segment $\mathbf{n}:=[0,n]=\{0,1,\dots,n\}$, $n\in\bn$, let $\pi_i:=(i\,\,i+1)\in\bp_\mathbf{n}$ $(i=0,\dots,n-1)$ be an adjacent transposition of $\mathbf{n}$. Then, the $\bc$-linear extension of the map
\[
\pi_i(b_0\ot\cdots\ot b_{i}\ot b_{i+1}\ot\dots\ot b_n):=u(b_i,b_{i+1})\,b_0\ot\cdots\ot b_{i+1}\ot b_i\ot\dots\ot b_n\,,
\]
defined on the elementary tensor products of homogeneous elements, isometrically extends to an element of $\mathsf{Aut}(\ga_n)$. This is just a particular case of a more general fact, as the following proposition shows. As a premise, fix $\rho\in\bp_\mathbf{n}$ and let $\ci_\rho:=\{(l,k)\in\mathbf{n}\times\mathbf{n}\colon l<k,\rho(l)>\rho(k)\}$ be the set of \emph{inversions of $\rho$}. Its cardinality is $\mathsf{inv}(\rho):=|\ci_\rho|\in\left\{0,\dots,\binom{n+1}{2}\right\}$. Also notice that if $\Sigma\colon\mathbf{n}\times\mathbf{n}\to\mathbf{n}\times\mathbf{n}$ is the switch bijection, $\ci_{\rho^{-1}}=(\rho\times\rho)\circ\Sigma(\ci_\rho)$. We then have the following
\begin{lem}
Each $\rho\in\bp_{\mathbf{n}}$ induces an element of $\mathsf{Aut}(\ga_n)$. Moreover, for every homogeneous $b_l\in\gb_o$,
\begin{equation}\label{permj}
\rho(b_0\ot\cdots\ot b_n)=\Big(\prod\limits_{(l,k)\in\ci_{\rho^{-1}}}u(b_l,b_k)\Big)\,b_{\rho(0)}\ot\dots\ot b_{\rho(n)}\,.
\end{equation}
\end{lem}
\begin{proof}
Since the $n$ adjacent transpositions of $\mathbf{n}$ form a generating set of the group $\bp_\mathbf{n}\sim S_{n+1}$, by the previous discussion, each $\rho\in\bp_{\mathbf{n}}$ induces an element of $\mathsf{Aut}(\ga_n)$. As concerns the formula, it suffices to prove it for any product $\rho$ of $N$ adjacent transpositions, with $N\geq1$. We shall do it by induction on $N$. For $N=1$, i.e. $\rho=(i\,\,i+1)$ for some $i=0,\dots,n-1$, $\ci_{\rho^{-1}}=\ci_\rho=\{(i,i+1)\}$ and the formula reduces to the one exposed above. Let us suppose that the result holds for some $N\geq1$ and prove it for $N+1$. If $\rho$ is a product of $N$ adjacent transpositions, then for every $i=0,\dots,n-1$, by inductive hypothesis we have
\begin{multline*}
((i\,\,i+1)\circ\rho)(b_0\ot\cdots\ot b_n)=\\=\prod\limits_{(l,k)\in\ci_{\rho^{-1}}}u(b_l,b_k)\,u(b_{\rho(i)},b_{\rho(i+1)})(b_{\rho(0)}\ot\cdots\ot b_{\rho(i+1)}\ot b_{\rho(i)}\ot\dots\ot b_{\rho(n)})\,.
\end{multline*}
Here, we are using the standard convention of reading the composition of cycles \emph{from right to left}, as for general functions. We are done if we prove the following equality:
\[
u(b_{\rho(i)},b_{\rho(i+1)})\prod\limits_{(l,k)\in\ci_{\rho^{-1}}}u(b_l,b_k)=\prod\limits_{(x,y)\in\ci_{(\rho\circ(i\,\,i+1))^{-1}}}u(b_x,b_y)
\]
or, equivalently,
\begin{equation}\label{eq}
u(b_{\rho(i)},b_{\rho(i+1)})\prod\limits_{(l,k)\in\ci_{\rho^{-1}}}u(b_l,b_k)\prod\limits_{(x,y)\in\ci_{(i\,\,i+1)\circ\rho^{-1}}}\overline{u(b_x,b_y)}=1
\end{equation}
(recall that a composition of two permutations on an elementary tensor product acts on the indices by reversing the composition).

Firstly, suppose $\rho(i)<\rho(i+1)$. Then, $(x,y)\in\ci_{(i\,\,i+1)\circ\rho^{-1}}$ if and only if either $(x,y)=(\rho(i),\rho(i+1))$ or $(x,y)\in\ci_{\rho^{-1}}$ and satisfies one of the following seven cases:
\begin{enumerate}
\item $\rho^{-1}(y)<\rho^{-1}(x)<i$
\item $\rho^{-1}(y)<i<i+1=\rho^{-1}(x)$
\item $\rho^{-1}(y)<i<i+1<\rho^{-1}(x)$
\item $\rho^{-1}(y)=i<i+1<\rho^{-1}(x)$
\item $\rho^{-1}(y)=i+1<\rho^{-1}(x)$
\item $i+1<\rho^{-1}(y)<\rho^{-1}(x)$
\item $\rho^{-1}(y)<i=\rho^{-1}(x)$
\end{enumerate}
In other words, $\ci_{(i\,\,i+1)\circ\rho^{-1}}=\ci_{\rho^{-1}}\,\sqcup\,\{(\rho(i),\rho(i+1))\}$ and (\ref{eq}) follows. Instead, if $\rho(i+1)<\rho(i)$, then $\ci_{\rho^{-1}}=\ci_{(i\,\,i+1)\circ\rho^{-1}}\,\sqcup\,\{(\rho(i+1),\rho(i))\}$ and again (\ref{eq}) is satisfied.
\end{proof}
For each $n\in\bn$, if $\j_n\colon\ga_n\hookrightarrow\ga$ is the canonical embedding of $\ga_n$ into $\ga$ and $\rho\in\bp_\mathbf{n}$, $\j_n\circ\rho\colon\ga_n\to\ga$ extends to a well defined $*$-automorphism of $\ga$ by universal property of the $C^*$-inductive limit. Even more: there exists a unique group representation on $\ga$ of the finitary symmetric group $\bp_\bn$. Indeed, for each $i\in\bn$, let $\Phi_i\in\mathsf{Aut}(\ga)$ be the isometric extension to $\ga$ of the flip $\pi_{i}:=(i\,\,i+1)\in\mathsf{Aut}(\ga_{i+1})$. In particular, $\Phi_i$ is involutive. Firstly, there exists a unique representation of the free (non-abelian) group $\bbf_\bn$ over $\bn$
\[
\begin{aligned}
\Pi\colon\bbf_\bn&\to\mathsf{Aut}(\ga)\\
w&\mapsto\F_{i_1}\circ\cdots\circ\F_{i_n}
\end{aligned}
\]
where $i_1\cdots i_n$ is the (unique) reduced form of the word $w\in\bbf_\bn$. Since the $\Phi_i$'s are involutive, this group representation is not faithful. Secondly, we have the following
\begin{thm}\label{action}
Let $\rho\in\bp_{\bn}$ expressed (not uniquely) as a finite product of adjacent transpositions $\rho=\pi_{i_1}\pi_{i_2}\cdots\pi_{i_n}$, ($i_k\in\bn$, $k=1,\dots,n$). Then, the assignment $\rho\mapsto\g_\rho:=\F_{i_1}\circ\cdots\circ\F_{i_n}$ realises a well-defined, pointwise norm-continuous action $\bp_\bn\stackrel{\g}{\curvearrowright}\ga$. Moreover, $\gamma$ commutes with the diagonal action $\a$ of $G$ on $\ga$: $\a\circ\g=\g\circ\a$. The action $\gamma$ is faithful provided that $\gb\neq\bc$.
\end{thm}
\begin{proof}
It is a well-known fact that $\bp_\bn\sim(\bbf_\bn\,|\,R)$ where $R$ is the set of relations in $\bbf_\bn$
$$
\begin{cases}
i_n^2=1\\
i_ni_{n+k}=i_{n+k}i_n\,\,(k\geq2)\\
i_ni_{n+1}i_n=i_{n+1}i_ni_{n+1}
\end{cases}
$$
for every $n\in\bn$. Since the normal subgroup $N\subset\bbf_\bn$ generated by the relations above lies in $\ker(\Pi)$, then $\Pi$ passes to the quotient modulo $N$, yielding a well-defined representation $\widetilde{\Pi}$ of $\bp_{\bn}$ on $\ga$. It is easy to see that $\widetilde{\Pi}=\gamma$. Indeed, we already observed that the $\F_i$ are involutive and, easily, $\F_i\circ\F_{i+k}=\F_{i+k}\circ\F_i$ for every $k\geq2$. Lastly, the third relation is guaranteed by the Yang-Baxter equality, satisfied by $\Phi_o$ in Proposition \ref{flip1} on $\gb_o\ot \gb_o\ot\gb_o$
$$
(\F_o\ot\text{id}_{\gb_o})\circ(\text{id}_{\gb_o}\ot\F_o)\circ(\F_o\ot\text{id}_{\gb_o})=(\text{id}_{\gb_o}\ot\F_o)\circ(\F_o\ot\text{id}_{\gb_o})\circ(\text{id}_{\gb_o}\ot\F_o)\,,
$$
then extended to $\gb\ot\gb\ot\gb$. Since $\bp_\bn$ is discrete, $\gamma$ is clearly pointwise norm-continuous. That $\gamma$ commutes with $\alpha$ is easily verifiable on the total set of homogeneous localised elements of $\ga$. If $\gb\neq\bc$, $\gamma$ is also faithful. Indeed, if $\rho\in\bp_\bn$ and $\rho\neq\rm{id}_\bn$, there exists $k\in\bn$ s.t. $\rho(k)\neq k$. Let $b_k\neq\idd_\bn$ and $b_j=\idd_\bn$ for every $j\neq k$. Then, by \eqref{permj},
\[
\gamma_\rho(b_0\ot b_1\ot\dots\ot b_k\ot\dots)\neq b_0\ot b_1\ot\dots\ot b_k\ot\dots
\]
i.e. $\gamma_\rho\neq\rm{id}_\ga$.
\end{proof}

\section{Exchangeable processes and symmetric states}
\label{free}

As explained in \cite{CF2}, a quantum unital stochastic process (based on the index set $\bn$) with sample $C^*$-algebra $\gb$ can be viewed as a state on the unital free product $C^*$-algebra $\bigast_\bn\gb$, intertwining the corresponding symmetries. For example, an exchangeable stochastic process corresponds to a symmetric state, that is one which is invariant under the natural action of the group of all finite permutations of the natural numbers. We refer the previous mentioned papers for details and proofs. Here, we describe in some details the situation corresponding to symmetric states arising from an infinite twisted $C^*$-tensor product.
We start by fixing the notation used in this section. 

Indeed, let $\gb$ be a unital $C^*$-algebra, the so-called {\it algebra of samples} equipped with an action $\b$ of a (abelian as usual) compact group $G$, and a bicharacter $u$ on $\widehat{G}\times\widehat{G}$. Define the unital infinite free product $C^*$-algebra and the infinite twisted $C^*$-tensor product $\ga^{(u)}$ (the latter based, as usual, on the minimal norm) as
$\ga^{\rm free}:=\bigast\!{}_\bn\,\gb\,,\quad \ga^{(u)}:=\ot_\bn\,\gb$, respectively.

Notice that the group of permutations $\bp_\bn$ acts on $\ga^{\rm free}$ through its natural action on the reduced words $\idd_{\ga^{\rm free}}=\emptyset$ (the empty word), and
$w=b_{j_1}\ast\cdots\ast b_{j_n}$ (with different adjacent indices), as
\begin{equation}
\label{epfrfe}
\a_g(\emptyset)=\emptyset\,, \quad \a_g\big(b_{j_1}\ast\cdots\ast b_{j_n}\big)=b_{g(j_1)}\ast\cdots\ast b_{g(j_n)}\,.
\end{equation}
On the other hand, the natural action of $\bp_\bn$ on $\ga^{(u)}$ arising from the exchange of the indices in $\bn$ is well defined if and only if $u$ is hermitian as we have shown above.

By construction, for each $i\in\bn$ there is a natural $*$-monomorphism 
$$
\gb\ni b\mapsto\iota_i(b):=b_i\in\bigast\!{}_\bn\,\gb\,,
$$
where $b_i$ is meant as $b$ suited in the place $i$ in $\bigast\!{}_\bn\,\gb$ with, obviously, $\iota_i(\idd_\gb)=\emptyset$ for each $i$.\footnote{Notice that 
$\gb\sim\bc\idd_\gb\oplus\gb_\#$ as $C^*$-algebra, where $\gb_\#$ is a $C^*$-algebra without the identity.}  By construction, $*\text{-alg}\{\iota_i(\gb)\,;\,i\in\bn\}$ is dense in $\ga^{\text{free}}$.

For $i\in\bn$, let $\s_i:\gb\to\ga^{(u)}$ be a unital $*$-homomorphism. By the universal property of the free product,
there exists a unique unital $*$-homomorphism $\s:\bigast\!{}_\bn\,\gb\to\ga^{(u)}$ making commutative the associated diagrams:

\begin{equation}\label{univw}
\xymatrix{\gb\ar[r]^{\iota_i} \ar[d]_{\s_i} &
\ga^{(\text{free})} \ar[dl]^{\quad \s\,\,,\quad i\,\in\,\bn\,,} \\
\ga^{(u)}}
\end{equation}
that is $\sigma_i=\sigma\circ\iota_i$ for each $i\in\bn$.
Specialising the above analysis to our situation, we take the natural embeddings $\s_i$ of $\gb$ into $\ga^{(u)}$, $i\in\bn$,
$$
\sigma_i\colon b\mapsto\underbrace{\idd\ot\cdots\ot\idd}_{\textrm{$i$-times}}\ot\,b\ot\idd\cdots\,.
$$
Since $*\text{-alg}\{\s_i(\gb)\,;\,i\in\bn\}$ is dense in $\ga^{(u)}$, we get a $*$-epimorphism $\s:\ga^{\rm free}\to\ga^{(u)}$, and thus the twisted algebra is viewed as the quotient
\begin{equation}
\label{epfrfe1}
\ga^{(u)}\sim\ga^{\rm free}/{\rm Ker}(\s)
\end{equation}
of the free one.
We are now in position to define the class of symmetric states associated to every twisted tensor product chain $\ga^{(u)}$, independently on the properties enjoyed by the bicharacter $u$.

We start with the definition of the classes of states, and that of symmetric ones, on the free $C^*$-algebra $\ga^{\rm free}$ coming from those on the twisted tensor product $C^*$-algebra $\ga^{(u)}$ built through the fixed bicharacter $u$.
\begin{defin}
\label{free00}
The weak$^*$-compact, convex set of states $\cs^{(u)}\big(\ga^{\rm free}\big)$ consists of the states in $\cs\big(\ga^{\rm free}\big)$ coming from $\ga^{(u)}$:
$$
\cs^{(u)}\big(\ga^{\rm free}\big):=\s^{\rm t}\big(\cs\big( \ga^{(u)}\big)\big)\,,
$$
where $\s$ is the $*$-epimorphism in \eqref{univw}.

The weak$^*$-compact, convex set of states $\cs_\bp^{(u)}\big(\ga^{\rm free}\big)$ are the states in $\cs^{(u)}\big(\ga^{\rm free}\big)$ which are invariant under the natural action $\a$ of $\bp$ on 
$\ga^{\rm free}$ given by \eqref{epfrfe}.
\end{defin}

To end the present section, notice that $\f\in\cs^{(u)}\big(\ga^{\rm free}\big)$ does not imply that $\f\circ\a_g\in\cs^{(u)}\big(\ga^{\rm free}\big)$ for each $g\in\bp$, unless $u$ is hermitian. If, instead, $u$ is hermitian, then the group of all finite permutations $\bp$ directly acts on $\ga^{(u)}$ and, concerning the transpose then 
$$
\a_g^{\rm t}\big(\cs^{(u)}\big(\ga^{\rm free}\big)\big)=\cs^{(u)}\big(\ga^{\rm free}\big)\,,\quad g\in\bp\,. 
$$
\begin{rem}
\label{stsosot}
Since $\bp$ acts on both algebras $\ga^{\rm free}$ (i.e. \eqref{epfrfe}), and on $\ga^{(u)}$ if $u$ is hermitian (i.e. Section \ref{opytg}),
in this situation it is straightforward to check that we have
$\cs^{(u)}_\bp\big(\ga^{\rm free}\big):=\s^{\rm t}\big(\cs_\bp\big(\ga^{(u)}\big)\big)$.
\end{rem}
Summarising, it is certainly meaningful to investigate the symmetric states in $\cs^{(u)}\big(\ga^{\rm free}\big)$
(i.e. those arising form $\ga^{(u)}$) for every $\ga^{(u)}$. On the other hand, if the involved bicharacter $u$ is hermitian, such symmetric states can be directly studied on the model $\ga^{(u)}$ without invoking the free $C^*$-algebra.

\section{Symmetric states for twisted tensor products and their ergodic properties}
\label{ssttptep}

From now on, we shall consider a $C^*$-dynamical system $(\ga,\bp,\g)$, where 
$\ga:=\underset{n\in\bn}{\ot}\gb=\ga^{(u)}$
for the fixed hermitian bicharacter $u$ on $\widehat{G}\times\widehat{G}$. The finitary symmetric group $\bp$ acts on $\ga$ via $\g$ as described in Theorem \ref{action}. The symmetric states on
the twisted chain $\ga$ are those which are invariant under all automorphisms $\{\g_g\,;\,g\in\bp\}$.
In view of Section \ref{free}, we shall study the basic properties of such symmetric states. For such a purpose, we denote directly by $g(x):=\g_g(x)$, $x\in\ga$ the action of $\g\in\bp$ on $x\in\ga$ through the automorphism $\g_g$. 
In view of the following result, for the commutator and anti-commutator between elements of an algebra, we put
$\{a,b\}_\pm:=ab\pm ba$.
\begin{prop}[Commutation and Anticommutation Relations in $\pi_\omega(\ga)^\bp$]
\label{relations}
Let $\omega\in\cs_{\bp}(\mathfrak{A})$ and $x,y,a,b\in\ga$ homogeneous. Then,
\begin{itemize}
\item[(i)] $\left\{\pi_\omega(x)^{\bp},\pi_\omega(y)^{\bp}\right\}_\pm=\big(1\pm\overline{u(x,y)}\big)\pi_\omega(x)^{\bp}\pi_\omega(y)^{\bp}$,
\item[(ii)] $\cam\big\{\omega(a\{x,\hat g(y)\}_\pm b)\big\}=(1\pm\overline{u(x,y)})\,\overline{u(b,y)}\left\langle\pi_\omega(axb)^{\bp}\pi_\omega(y)^{\bp}\xi_\omega,\xi_\omega\right\rangle$.
\end{itemize}
In particular, if $\pi_\omega(x)^{\bp}$ is non-zero, then $u(x,x)=1$.
\end{prop}
\begin{proof}
(i) We start from homogeneous $x,y\in\ga_o$, respectively localised in the discrete segments $\mathbf{m}:=[0,m]$ and $\mathbf{n}:=[0,n]$, $m,n\in\bn$, and prove it for the anti-commutator. The case involving the commutator is completely similar. 
By the Mean Ergodic Theorem,
\begin{multline*}
\{E_\omega\pi_\omega(x)E_\omega,E_\omega\pi_\omega(y)E_\omega\}=\cam\{E_\omega\pi_\omega(\{x,\hat g(y)\})E_\omega\}\\=\slim\limits_{N\to+\infty}\frac{1}{(N+1)!}\sum\limits_{g\in\bp_\mathbf{N}}E_\omega\pi_\omega(\{x,g(y)\})E_\omega\,.
\end{multline*}
Let $v:=\max\{m,n\}$. For each $N\geq2v+1$, consider the family $\Gamma_{N,v}:=\{g\in\bp_{{\bf N}}\mid{\bf v}\cap\rho({\bf v})=\varnothing\}$ of permutations of $\mathbf{N}:=[0,N]$ which fully displace $\mathbf{v}$. Then,
\begin{multline*}
\sum_{g\in\bp_{\bf N}}E_\om\pi_\om\big(\{x,g(y)\}\big)E_\om\\=\sum_{g\in\G_{N,v}}E_\om\pi_\om\big(\{x,g(y)\}\big)E_\om
+\sum_{g\not\in\G_{N,v}}E_\om\pi_\om\big(\{x,g(y)\}\big)E_\om
\\=\big(1+\overline{u(x,y)}\big)\sum_{g\in\G_{N,v}}E_\om\pi_\om\big(xg(y)\big)E_\om+\sum_{g\not\in\G_{N,v}}E_\om\pi_\om\big(\{x,g(y)\}\big)E_\om\\
=\big(1+\overline{u(x,y)}\big)\sum_{g\in\bp_{\bf N}}E_\om\pi_\om\big(xg(y)\big)E_\om\\
+\sum_{g\not\in\G_{N,v}}\Big(E_\om\pi_\om\big(\{x,g(y)\}\big)E_\om-\big(1+\overline{u(x,y)}\big)E_\om\pi_\om\big(xg(y)\big)E_\om\Big)\\
=\big(1+\overline{u(x,y)}\big)\sum_{g\in\bp_{\bf N}}E_\om\pi_\om\big(xg(y)\big)E_\om+\sum_{g\not\in\G_{N,v}}E_\om\pi_\om\big(g(y)x-\overline{u(x,y)}xg(y)\big)E_\om\,.
\end{multline*}
By Lemma \ref{estimation}, the norm of the second addendum above is $o((N+1)!)$ as $N\geq2v+1$ tends to $+\infty$, hence
\begin{multline*}
\frac{1}{(N+1)!}\left\|\sum_{g\not\in\G_{N,v}}E_\om\pi_\om\big(g(y)x-\overline{u(x,y)}xg(y)\big)E_\om\right\|\\
\leq\frac{1}{(N+1)!}\sum_{g\not\in\G_{N,v}}\left\|g(y)x-\overline{u(x,y)}xg(y)\right\|\\\leq2\|x\|\|y\|\frac{|\G_{N,v}^c|}{(N+1)!}\leq2\|x\|\|y\|\frac{C_v}{N+1}\xrightarrow{N\uparrow+\infty}0\,.
\end{multline*}
Therefore, by passing to the limit as $N$ tends to $+\infty$,
\begin{multline*}
\{E_\omega\pi_\omega(x)E_\omega,E_\omega\pi_\omega(y)E_\omega\}=\big(1+\overline{u(x,y)}\big)\slim\limits_{N\to+\infty}\frac{1}{N!}\sum\limits_{g\in\bp_\mathbf{N}}E_\omega\pi_\omega(xg(y))E_\omega=\\=\big(1+\overline{u(x,y)}\big)\slim\limits_{N\to+\infty}\Big(\frac{1}{N!}\sum\limits_{g\in\bp_\mathbf{N}}E_\omega\pi_\omega(x)U_\omega(g)\pi_\omega(y)E_\omega\Big)=\\=\big(1+\overline{u(x,y)}\big)E_\omega\pi_\omega(x)E_\omega\pi_\omega(y)E_\omega
\end{multline*}
and equality (i) is established for localized homogeneous elements $x,y\in\ga_\infty$. Now, since for every $\sigma\in\widehat{G}$, $\ga_\sigma=\overline{(\ga_\infty)_\sigma}^\ga$, if $x,y\in\ga$ are any pair of homogeneous elements and $\varepsilon\in(0,1)$, there exist $x_\varepsilon,y_\varepsilon\in\ga_\infty$ s.t. $\partial x_\varepsilon=\partial x$, $\partial y_\varepsilon=\partial y$ and $\|x-x_\varepsilon\|_\ga,\|y-y_\varepsilon\|_\ga<\varepsilon$. It follows that
\[
\|E_\omega\pi_\omega(x)E_\omega\pi_\omega(y)E_\omega-E_\omega\pi_\omega(x_\varepsilon)E_\omega\pi_\omega(y_\varepsilon)E_\omega\|<\varepsilon(\|x\|+\|y\|+\varepsilon)\,,
\]
i.e.
\[
\left\|\{E_\omega\pi_\omega(x)E_\omega,E_\omega\pi_\omega(y)E_\omega\}-\{E_\omega\pi_\omega(x_\varepsilon)E_\omega,E_\omega\pi_\omega(y_\varepsilon)E_\omega\}\right\|<2\varepsilon(\|x\|+\|y\|+\varepsilon)\,.
\]
By arbitrariness of $\varepsilon\in(0,1)$, equality (i) is then established for any homogeneous elements $x,y\in\ga$. A similar proof holds for (ii).

As concerns (ii), firstly observe that for every $a,b,x,y\in\ga$, both the nets
\[
\bigg(\frac{1}{|\bp_S|}\sum\limits_{g\in\bp_S}\om(axbg(y))\bigg),\quad
\bigg(\frac{1}{|\bp_S|}\sum\limits_{g\in\bp_S}\om(g(y)axb)\bigg)
\]
have each a unique cluster point, given by
\[
\cam\big\{\om(axb\hat g(y))\big\}=\langle\pi_\om(axb)E_\om\pi_\om(y)\xi_\om,\xi_\om\rangle\,,
\]
\[
\cam\big\{\om(\hat g(y)axb)\big\}=\langle\pi_\om(y)E_\omega\pi_\om(axb)\xi_\om,\xi_\om\rangle\,.
\]
Now, if $a,b,x,y$ are homogeneous and belonging to $\ga_\infty$, by reasoning as above we get that $\cam\big\{\om(ax\hat g(y)b)\big\}$ exists too, as
\begin{equation}\label{aver1}
\cam\big\{\om(ax\hat g(y)b)\big\}=\overline{u(b,y)}\cam\big\{\om(axb\hat g(y))\big\}\,.
\end{equation}
To extend \eqref{aver1} to every homogeneous element in $\ga$, we observe that for any $\varepsilon\in(0,1)$, there exist $a_\varepsilon,x_\varepsilon,y_\varepsilon,b_\varepsilon\in\ga_\infty$ s.t. $\partial a_\varepsilon=\partial a,\partial x_\varepsilon=\partial x,\partial y_\varepsilon=\partial y,\partial b_\varepsilon=\partial b$ and $|\omega(ax g(y)b)-\omega(a_\varepsilon x_\varepsilon g(y_\varepsilon)b_\varepsilon)|<\varepsilon$.
In particular, 
$$
\bigg|\frac{1}{N!}\sum\limits_{ g\in\bp_{\mathbf{N}}}\omega(ax g(y)b)-\frac{1}{N!}\sum\limits_{ g\in\bp_{\mathbf{N}}}\omega(a_\varepsilon x_\varepsilon g(y_\varepsilon)b_\varepsilon)\bigg|<\varepsilon\,,\,\,\,N\geq1\,,
$$
Since $\bigg(\frac{1}{N!}\sum\limits_{ g\in\bp_{\mathbf{N}}}\omega(a_\varepsilon x_\varepsilon g(y_\varepsilon)b_\varepsilon)\bigg)_N$ converges, $\bigg(\frac{1}{N!}\sum\limits_{ g\in\bp_{\mathbf{N}}}\omega(ax g(y)b)\bigg)_N$ 
is a Cauchy sequence in $\bc$, thus convergent as well.
It means that $\cam\big\{\omega(ax\hat g(y)b)\big\}$ exists and 
$$
\cam\big\{\omega(ax\hat g(y)b)\big\}=\lim\limits_{\varepsilon\downarrow0^+}\cam\big\{\omega(a_\varepsilon x_\varepsilon \hat g(y_\varepsilon)b_\varepsilon)\big\}\,.
$$ 
By analogously approximating $\cam\big\{\omega(axb\hat g(y))\big\}$, we get \eqref{aver1} for any homogeneous elements $a,x,y,b\in\ga$. Similarly,
\begin{equation}\label{aver2}
\cam\big\{\om(a\hat g(y)xb)\big\}=u(a,y)\cam\big\{\om(\hat g(y)axb)\big\}
\end{equation}
By \eqref{aver1} and \eqref{aver2},
\begin{align*}
\cam\big\{\om(ax\hat g(y)b)\big\}=&\overline{u(b,y)}\cam\big\{\om(axb\hat g(y))\big\}\\
=&\overline{u(b,y)}\langle\pi_\om(axb)E_\om\pi_\om(y)\xi_\om,\xi_\om\rangle\\
\cam\big\{\om(a\hat g(y)xb)\big\}=&u(a,y)\cam\big\{\om(\hat g(y)axb)\big\}\\
=&\overline{u(x,y)}\,\overline{u(b,y)}\langle\pi_\om(axb)E_\om\pi_\om(y)\xi_\om,\xi_\om\rangle
\end{align*}
where we used point (i). Therefore, 
\[
\cam\big\{\om(a\{x, \hat g(y)\}_\pm b)\big\}=(1\pm\overline{u(x,y)})\,\overline{u(b,y)}\langle\pi_\om(axb)E_\om\pi_\om(y)\xi_\om,\xi_\om\rangle
\]
that is (ii).

Lastly, If $E_\omega\pi_\omega(x)E_\omega$ is non-zero, there exists $\xi\in\ch_\om$ such that
\[
\big\|E_\omega\pi_\omega(x)E_\omega\xi\big\|^2>0\,.
\]
Now,
by exploiting (i) with $y=x^*$ and noticing that $u(x,x)$ is real, we get
$$
0<\big\|E_\omega\pi_\omega(x)E_\omega\xi\big\|^2=u(x,x)\big\|E_\omega\pi_\omega(x^*)E_\omega\xi\big\|^2\,,
$$
which leads to $u(x,x)=1$.
\end{proof}
Thanks to Proposition \ref{relations}, we can give a necessary and sufficient condition for a symmetric state $\omega\in\cs_{\bp}(\ga)$ to be $\bp$-abelian, a crucial property for the upcoming De Finetti Theorem, e.g. \cite{St2, Ffe}. 

For $\omega\in\cs_\bp(\ga)$, let
\[
\mathrm{supp}(\omega):=[\ga_\sigma\,;\,\omega\lceil_{\ga_\sigma}\neq0]=[\ga_\sigma\,;\,\pi_\omega(\ga_\sigma)\xi_\omega\not\perp\xi_\omega]\subset\ga\,.
\]
Since $u$ is hermitian, $u(\sigma,\sigma)\in\{\pm1\}$ for every $\sigma\in\widehat{G}$. Therefore, recall that
\begin{equation}
\label{annis}
\Delta_{\pm}:=\{\sigma\in\widehat{G}\,;\, u(\sigma,\sigma)=\pm1\}\,.
\end{equation}
Evidently, $\Delta_+$ is a subgroup of $\widehat{G}$, called \textit{isotropy subgroup}. Its annihilator is
$$
\Delta_+^\perp:=\{g\in G\,;\,\sigma(g)=1,\,\sigma\in\Delta_+\}\subset G\,.
$$ 
Also, put $\Delta_+^*:=\Delta_+\setminus\{0\}\subset\widehat{G}$. If $\sigma\in\Delta_+^*$, let 
\[
T_\sigma:=\{\tau\in\Delta^*_+\,;\, u(\sigma,\tau)\neq1\}\subset\widehat{G}\,.
\]
Then,
\begin{itemize}
\item $T_\sigma$ is a (possibly, empty) symmetric subset of $\Delta_+^*$, i.e. $T_\sigma=-T_\sigma$
\item $T_\sigma=T_{-\sigma}$,
\item $\tau\in T_\sigma$ if and only if $\sigma\in T_\tau$.
\end{itemize}
In particular, $\pi_\omega\Big(\dotplus_{\tau\in T_\sigma}\ga_{\tau}\Big)^{\bp}$ is a selfadjoint operator space (i.e. an operator system) in $\cb(E_\om\ch_\omega)$. If $T_\sigma=\varnothing$, then $\dotplus_{\tau\in T_\sigma}\ga_{\tau}$ is set to be $\{0\}$.

We are now ready for the investigation of the ergodic properties of $\cs_\bp(\ga)$. It will turn out that every symmetric state satisfies an invariance property.
\begin{thm}[Ergodic properties of symmetric states]\label{erg}
Let $\omega\in\cs_\bp(\mathfrak{A})$. It results that:
\begin{itemize}
\item[(i)] $\mathrm{supp}(\omega)\subset[\ga_\sigma\,;\,\pi_\omega(\ga_\sigma)^{\bp}\neq\{0\}]
\subset[\ga_\sigma\,;\,\sigma\in\Delta_+]$ and, in particular, $\omega$ is $\Delta_+^\perp$-invariant;
\item[(ii)] $\omega$ is $\bp$-\emph{abelian} if and only if for each $\sigma\in\Delta_+^*$,
\begin{equation}\label{abel}
\overline{\pi_\omega\left(\ga_{\sigma}\right)^{\bp}E_\om\ch_\omega}\perp\overline{\pi_\omega\left(\dotplus_{\tau\in T_\sigma}\ga_\tau\right)^{\bp}E_\om\ch_\omega}\,.
\end{equation}
\end{itemize}
\end{thm}
\begin{proof} Given a homogeneous $x\in\mathrm{supp}(\om)$, we have $0<|\om(x)|\leq\|E_\om\pi_\om(x)E_\om\|$. Therefore, by Proposition \ref{relations}, $u(x,x)=1$ i.e. $x\in[\ga_\sigma\colon\sigma\in\Delta_+]$ and point (i) is accomplished. As concerns point (ii) in the claim, point (i) along with an easy application of point (ii) in Proposition \ref{relations}, implies that for homogeneous $x,y\in\ga$, $\left[\pi_\omega(x)^\bp,\pi_\omega(y)^\bp\right]=0$ whenever $\partial x\in\Delta_-$, $\partial y\in\Delta_-$ or $u(x,y)=1$. That allows us to reduce to the case where $\partial x\in\Delta_+^*$ and $\partial y\in T_{\partial x}$. In this situation, again by point (ii) in Proposition \ref{relations}, $\left[\pi_\omega(x)^\bp,\pi_\omega(y)^\bp\right]=0$ if and only if $\pi_\omega(x)^\bp\pi_\omega(y)^\bp=0$, that is
\[
\langle\pi_\omega(x^*)^\bp\xi,\pi_\omega(y)^\bp\eta\rangle_{\ch_\omega^\bp}=0,\quad\xi,\eta\in\ch_\omega^\bp
\]
or $\text{ran}\left(\pi_\omega(x)^\bp\right)=\text{ran}\left(\pi_\omega(x^*)^\bp\right)\perp\text{ran}\left(\pi_\omega(y)^\bp\right)$, whence \eqref{abel} follows. 
\end{proof}
\begin{rem}\label{choquet}
In view of Corollary 4.4 in \cite{Bat}, every $\omega\in\cs_\bp(\ga)$ is $\bp$-abelian if and only if $\cs_\bp(\ga)$ is a Choquet simplex, in which case, for every $\om\in\cs_\bp(\ga)$ there exists a unique probability Radon measure $\n_\om$ on $\cs_\bp(\ga)$ satisfying:
\begin{itemize}
\item[(i)] if $B\subset\cs_\bp(\ga)$ is a Baire set with $B\cap\ce_\bp(\ga)=\emptyset$, then $\n_\om(B)=0$;
\item[(ii)] $\om=\int_{\cs_\bp(\ga)}\f\,\di\n_\om(\f)$.
\end{itemize}
\end{rem}
The measure $\n_\om$ is maximal, in the sense explained in \cite{S}, pag. 123, among all measures satisfying (ii) above and is nothing else than the measure associated to 
$\{\pi_\om(\ga),U_\om(\bp)\}'\subset\pi_\om(\ga)'$, see \cite{S}, Proposition 3.1.5.

We also note that, since $\n_\om$ satisfies (i) above, it is said that such a measure is {\it pseudo-supported} on $\ce_\bp(\ga)$. When $\ga$ is separable, $\n_\om$ is indeed supported on $\ce_\bp(\ga)$ because any Baire set is also a Borel set. 
Therefore,
\begin{equation}
\label{bary}
\om=\int_{\ce_\bp(\ga)}\f\,\di\n_\om(\f)\,,
\end{equation}
where \eqref{bary} is the so-called {\it barycentric decomposition} of $\om\in\cs_\bp(\ga)$.

The above discussion allows us to provide a sufficient condition in order to get the version of De Finetti Theorem for general twisted systems based on hermitian bicharacters.
\begin{cor}\label{tpfekl}
If $u\lceil_{\Delta_+\times \Delta_+}\equiv1$, then $\cs_\bp(\ga)$ is a Choquet simplex.
\end{cor}
\begin{proof}
If $u\lceil_{\Delta_+\times \Delta_+}\equiv1$ then, for each $\sigma\in\Delta_+^*$, $T_\sigma=\varnothing$. Hence, (iii) in \eqref{erg} is trivially satisfied for every $\omega\in\cs_\bp(\ga)$ since $\dotplus_{\tau\in T_\sigma}\ga_{\tau}=\{0\}$. By Remark \ref{choquet}, $\cs_\bp(\ga)$ is a Choquet simplex.
\end{proof}
To end the present section, we point out the following surprising fact. The hypoythesis required in Corollary \ref{tpfekl} provides us only the three cases listed in Theorem \ref{three}. The first two cases, $G=\{1\}$ and $G=\bz_2$, correspond to the well-known models of usual tensor product and Fermi product, respectively. The third case, involving the Klein group $G=\bz_2\times\bz_2$, has never been addressed before. In all cases, the single states appearing in the infinite product providing the extremal symmetric states should be invariant under the (possibly proper) closed subgroup
$\D_+^\perp$ of $G$, and not w.r.t. the whole.

\section{The Klein twisted $C^*$-chain}
\label{klsce}

As we shown in the previous sections, in order to carried on a fruitful investigation of the properties of symmetric states, the analysis reduces to three cases only. The case of the usual tensor product is well known. It corresponds to the trivial case $G=\{1\}$, and thus $\widehat{G}=\{0\}$. The second case corresponds to the Fermi tensor product, only recently studied in \cite{Ffe}. In this case, $G$ is the multiplicative group $\bz_2$ with, again, $\widehat{G}=\bz_2$ (the group operation coinciding with the addition mod-2). In both cases, the isotropy subgroup $\D_+\subset\widehat{G}$ is trivial, and then $\D_+^\perp$ is the whole group. The last case, corresponding to the Klein group $K_4=\bz_2\times\bz_2$ (and dual group $\widehat{K_4}=\bz_2\oplus\bz_2$), deserves a more accurate analysis since $\D_+\subsetneq\widehat{K_4}$. 

Consider then a $C^*$-dynamical system $(\gb,K_4,\beta)$ based on the Klein group, together with the (unique, up to equivalence) Klein bicharacter $u_{\rm{K}}$.
Since the isotropy subgroup 
$\D_+$, and correspondingly its annihilator subgroup $\D_+^\perp\subsetneq K_4$, play a crucial role in investigating the ergodic properties of the model (cf. Theorem \ref{erg}), we immediately recognise that, in all nontrivial cases, $\cs_{\D_+^\perp}(\ga)\supsetneq\cs_{K_4}(\ga)$. This would imply that the chain built by using the min-norm might be different from (precisely, strictly dominated by) that built from $\cs_{\D_+^\perp}(\ga)$. This possible obstruction can be avoided by assuming that the involved $C^*$-algebra is nuclear.
As we will see below, a surprising fact will be that the extremal symmetric states are product states of a single state, invariant under $\D_+^\perp$ only, instead of the whole Klein group. 

We start with the following
\begin{lem}\label{kleinpos}
Let $(\gb,K_4,\beta)$ be a $C^*$-dynamical system and consider $u_{\rm{K}}\in\mathsf{Bic}(K_4)$. If $\omega,\varphi\in\cs(\gb)$ and $\mathrm{supp}(\omega),\mathrm{supp}(\varphi)\subset\gb_{(0,0)}\oplus\gb_{(1,1)}$ (equivalently, $\omega,\varphi\in\cs_{\Delta^\perp_+}(\gb)$), then $\omega\times\varphi$ is a state on $\gb\ok\gb$ and
\[
|(\omega\times\varphi)(x)|\leq(\omega\times\varphi)(x^*x)^{1/2},\quad x\in\gb\ok\gb\,.
\]
In particular, $\pi_{\omega\times\varphi}$ uniquely extends to a representation of $\gb\ok\gb$ acting by bounded operators on the Hilbert space $\ch_{\omega\times\varphi}$.
\end{lem}
\begin{proof}
Let $x:=\sum\limits_{i=1}^na^{(i)}\odot b^{(i)}=\sum\limits_{i=1}^n\left(a_{+}^{(i)}+a_{-}^{(i)}\right)\odot\left(b_{+}^{(i)}+b_{-}^{(i)}\right)\in\gb\ok\gb$, where $a_+^{(i)},b_+^{(i)}\in\gb_{(0,0)}\oplus\gb_{(1,1)}$ and $a_-^{(i)},b_-^{(i)}\in\gb_{(0,1)}\oplus\gb_{(1,0)}$. Then,
\[
x^*x=\sum\limits_{i,j=1}^n\left[\left(a_{+}^{(i)}+a_{-}^{(i)}\right)\odot\left(b_{+}^{(i)}+b_{-}^{(i)}\right)\right]^*\left[\left(a_{+}^{(j)}+a_{-}^{(j)}\right)\odot\left(b_{+}^{(j)}+b_{-}^{(j)}\right)\right]
\]
where, for every $i=1,\dots,n$,
\begin{multline*}
\left[\left(a_{+}^{(i)}+a_{-}^{(i)}\right)\odot\left(b_{+}^{(i)}+b_{-}^{(i)}\right)\right]^*\\
=\left(a_{+}^{(i)}\odot b_{+}^{(i)}\right)^\dagger+\sum\limits_{g\in\Delta^\perp_{-}}\left(\beta_{g}(a^{(i)})\odot b_g^{(i)}\right)^\dagger+\sum\limits_{g\in\Delta^\perp_{+}}\left(\beta_{g}(a^{(i)}_{-})\odot b_g^{(i)}\right)^\dagger\,.
\end{multline*}
Hence, for each $i,j=1,\dots,n$, once set $\Delta_{-}^\perp:=\{(0,1),(1,0)\}$,
\begin{multline*}
\left[\left(a_{+}^{(i)}+a_{-}^{(i)}\right)\odot\left(b_{+}^{(i)}+b_{-}^{(i)}\right)\right]^*\left[\left(a_{+}^{(j)}+a_{-}^{(j)}\right)\odot\left(b_{+}^{(j)}+b_{-}^{(j)}\right)\right]\\
=(a_{+}^{(i)}\odot b_{+}^{(i)})^\dagger\cdot(a_{+}^{(j)}\odot b^{(j)})+\sum\limits_{g\in\Delta^\perp_{-}}(\beta_{g}(a^{(i)})\odot b_g^{(i)})^\dagger\cdot(\beta_{g}(a_{+}^{(j)})\odot b^{(j)})\\+\sum\limits_{g\in\Delta^\perp_{+}}(\beta_{g}(a^{(i)}_{-})\odot b_g^{(i)})^\dagger\cdot(a^{(j)}_{+}\odot b^{(j)})+\sum\limits_{g\in\Delta^\perp_{+}}(a_{+}^{(i)}\odot b_g^{(i)})^\dagger\cdot(\beta_{g}(a_{-}^{(j)})\odot b^{(j)})\\+\sum\limits_{g\in\Delta^\perp_{-}}(\beta_{g}(a^{(i)})\odot b_g^{(i)})^\dagger\cdot(\beta_{g}(a_{-}^{(j)})\odot b^{(j)})+\sum\limits_{g\in\Delta^\perp_{+}}(\beta_{g}(a^{(i)}_{-})\odot b_g^{(i)})^\dagger\cdot(\beta_{g}(a^{(j)}_{-})\odot b^{(j)})
\end{multline*}
and, taking into account that $\mathrm{supp}(\omega),\mathrm{supp}(\varphi)\subset\gb_{(0,0)}\oplus\gb_{(1,1)}$,
\begin{multline}\label{prod}
\left(\omega\times\varphi\right)
\left(
\left[\left(a_{+}^{(i)}+a_{-}^{(i)}\right)\odot\left(b_{+}^{(i)}+b_{-}^{(i)}\right)\right]^*
\left[\left(a_{+}^{(j)}+a_{-}^{(j)}\right)\odot\left(b_{+}^{(j)}+b_{-}^{(j)}\right)\right]
\right)\\
=\psi_{\omega,\varphi}\left[\left(a_{+}^{(i)}\odot b_{+}^{(i)}\right)^\dagger\cdot\left(a_{+}^{(j)}\odot b_{+}^{(j)}\right)\right]\\+
\psi_{\omega,\varphi}\left[\sum\limits_{g\in\Delta^\perp_{-}}\left(\beta_{g}\left(a^{(i)}_{+}\right)\odot b_g^{(i)}\right)^\dagger\cdot\left(\beta_{g}\left(a_{+}^{(j)}\right)\odot b_{-}^{(j)}\right)\right]\\+\psi_{\omega,\varphi}\left[\sum\limits_{g\in\Delta^\perp_{-}}\left(\beta_{g}\left(a^{(i)}_{-}\right)\odot b_g^{(i)}\right)^\dagger\cdot\left(\beta_{g}\left(a_{-}^{(j)}\right)\odot b^{(j)}_{-}\right)\right]\\+
\psi_{\omega,\varphi}
\left[
\left(a^{(i)}_{-}\odot b_{+}^{(i)}\right)^\dagger\cdot\left(a^{(j)}_{-}\odot b^{(j)}_{+}\right)\right]\,.
\end{multline}
Now, for a fixed representative $\widetilde{g}\in\Delta_{-}^\perp$,
\begin{itemize}
\item $\beta_{g}\left(a_{+}^{(i)}\right)=\beta_{\widetilde{g}}\left(a_{+}^{(i)}\right)$ for $g\in\Delta^\perp_{-}$. Therefore, the second addend in \eqref{prod} becomes
\begin{multline*}
\psi_{\omega,\varphi}\left[\sum\limits_{g\in\Delta^\perp_{-}}\left(\beta_{g}(a^{(i)}_{+})\odot b_g^{(i)}\right)^\dagger\cdot\left(\beta_{g}(a_{+}^{(j)})\odot b_{-}^{(j)}\right)\right]\\
=\psi_{\omega,\varphi}\left[\sum\limits_{g\in\Delta^\perp_{-}}\left(\beta_{\widetilde{g}}(a^{(i)}_{+})\odot b_{g}^{(i)}\right)^\dagger\cdot\left(\beta_{\widetilde{g}}(a_{+}^{(j)})\odot b_{-}^{(j)}\right)\right]\\
=\psi_{\omega,\varphi}\left[\left(\sum\limits_{g\in\Delta^\perp_{-}}\beta_{\widetilde{g}}(a^{(i)}_{+})\odot b_{g}^{(i)}\right)^\dagger\cdot\left(\beta_{\widetilde{g}}(a_{+}^{(j)})\odot b_{-}^{(j)}\right)\right]\\=\psi_{\omega,\varphi}\left[\left(\beta_{{g'}}\left(a_{+}^{(i)}\right)\odot b_{-}^{(i)}\right)^\dagger\cdot\left( \beta_{{g'}}\left(a_{+}^{(j)}\right)\odot b_{-}^{(j)}\right)\right]\,.
\end{multline*}
\item $\beta_{g}\left(a_{-}^{(i)}\right)=f_{g-\widetilde{g}}\beta_{\widetilde{g}}\left(a_{-}^{(i)}\right)$ for $g\in\Delta^\perp_{-}$, with $f_{g-\widetilde{g}}\in\{\pm1\}$. Therefore, the third addend in \eqref{prod} becomes
\begin{multline*}
\sum\limits_{g\in\Delta^\perp_{-}}\left(f_{g-\widetilde{g}}\beta_{\widetilde{g}}\left(a_{-}^{(i)}\right)\odot b_g^{(i)}\right)^\dagger\cdot\left(f_{g-\widetilde{g}}\beta_{\widetilde{g}}\left(a_{-}^{(j)}\right)\odot b^{(j)}_{-}\right)\\=\sum\limits_{g\in\Delta^\perp_{-}}(f_{g-\widetilde{g}})^2\left(\beta_{\widetilde{g}}\left(a_{-}^{(i)}\right)\odot b_g^{(i)}\right)^\dagger\cdot\left(\beta_{\widetilde{g}}\left(a_{-}^{(j)}\right)\odot b^{(j)}_{-}\right)\\
=\left(\beta_{\widetilde{g}}\left(a_{-}^{(i)}\right)\odot b_{-}^{(i)}\right)^\dagger\cdot\left(\beta_{\widetilde{g}}\left(a_{-}^{(j)}\right)\odot b^{(j)}_{-}\right)
\end{multline*}
\end{itemize}
Putting all together (once fixed $\widetilde{g}\in\Delta_{-}^\perp$),
\begin{multline*}
(\omega\times\varphi)(x^*x)=\psi_{\omega,\varphi}\left[\left(\sum\limits_{i=1}^na_{+}^{(i)}\odot b_{+}^{(i)}\right)^\dagger\cdot\left(\sum\limits_{j=1}^na_{+}^{(j)}\odot b_{+}^{(j)}\right)\right]\\
+\psi_{\omega,\varphi}\left[\left(\sum\limits_{i=1}^n\beta_{\widetilde{g}}\left(a_{+}^{(i)}\right)\odot b_{-}^{(i)}\right)^\dagger\cdot\left(\sum\limits_{j=1}^n\beta_{\widetilde{g}}\left(a_{+}^{(j)}\right)\odot b_{-}^{(j)}\right)\right]\\
+\psi_{\omega,\varphi}\left[\left(\sum\limits_{i=1}^n\beta_{\widetilde{g}}(a^{(i)}_{-})\odot b_{-}^{(i)}\right)^\dagger\cdot\left(\sum\limits_{j=1}^n\beta_{\widetilde{g}}(a_{-}^{(j)})\odot b^{(j)}_{-}\right)\right]\\+\psi_{\omega,\varphi}\left[\left(\sum\limits_{i=1}^na^{(i)}_{-}\odot b_{+}^{(i)}\right)^\dagger\cdot\left(\sum\limits_{j=1}^na^{(j)}_{-}\odot b^{(j)}_{+}\right)\right]
\end{multline*}
All the four terms are manifestly positive, whence $\omega\times\varphi$ is a state on $\gb\ok\gb$. The inequality in the assertion is nothing but the Cauchy–Bunyakovsky–Schwarz inequality. Lastly, by Lemma 3.1 in \cite{FV}, $\pi_{\omega\times\varphi}$ uniquely extends to a representation of $\gb\ok\gb$ acting by bounded operators on the Hilbert space $\ch_{\omega\times\varphi}$.
\end{proof}
In view of Lemma \ref{kleinpos}, we can build an ``intermediate" $C^*$-norm on $\gb\ok\gb$ based on products of $\Delta_+^\perp$-invariant states, which is compatible with the direct product action of $K_4\times K_4$. 
\begin{thm}\label{nucl}
Let $(\gb,K_4,\beta)$ be a $C^*$-dynamical system. Then,
\[
\|\cdot\|_{+}:=\sup\limits_{\omega,\varphi\in\cs_{\Delta_+^\perp}(\gb)}\|\pi_{\omega\times\varphi}(\cdot)\|
\]
is a $(\beta\times\beta)$-compatible $C^*$-norm on $\gb\ok\gb$.\\In particular, $\|x\|_{\min}\leq\|x\|_{+}\leq\|x\|_{\max}$ $(x\in\gb\ok\gb)$, where the equalities are simultaneously satisfied for every $x\in\gb\ok\gb$ if and only if $\gb^{K_4}$ is nuclear.
\end{thm}
\begin{proof} Evidently $\cs_{K_4}(\gb)\subset\cs_{\Delta_+^\perp}(\gb)$, hence $\cs_{\Delta_+^\perp}(\gb)\times\cs_{\Delta_+^\perp}(\gb)$ separates the points of $\gb\ok\gb$ and $\|\cdot\|_{\Delta_+^\perp}$ defines a $C^*$-norm on $\gb\ok\gb$. Now, if $\omega\in\cs_{\Delta_+^\perp}(\gb)$, then $\omega\circ\beta_g\in\cs_{\Delta_+^\perp}(\gb)$ for every $g\in K_4$. It follows that $\cs_{\Delta_+^\perp}(\gb)\times\cs_{\Delta_+^\perp}(\gb)$ is left globally stable by the transposed action $(\beta\times\beta)^{\rm t}$ of $K_4\times K_4$. By Theorem 11.1 in \cite{FV}, $\|\cdot\|_{+}$ is $(\beta\times\beta)$-compatible. Lastly, by Theorem 13.1 in \cite{FV} and the very definition of the maximal $C^*$-norm,
\[
\|x\|_{\min}\leq\|x\|_{+}\leq\|x\|_{\max},\quad x\in\gb\ok\gb\,,
\]
where the equalities are simultaneously satisfied for every $x\in\gb\ok\gb$ if and only if $\gb^{K_4}$ is nuclear, thanks to Theorem 17.1 in \cite{FV}.
\end{proof}
Without assuming the nuclearity of $\gb$, it seems unclear whether $\|\cdot\|_{\min}$ coincides with $\|\cdot\|_{+}$ or not, since in general $\cs_{\Delta_+^\perp}(\gb)$ can properly contain $\cs_{K_4}(\gb)$. Still, when $\gb$ is nuclear, in view of Lemma \ref{kleinpos} and Theorem \ref{nucl}, we can construct the \emph{infinite product state} of a Klein twisted chain $(\ga,K_4,\alpha)$ generated by a sequence $(\psi_i)_{i\in\bn}\subset\cs_{\Delta_+^\perp}(\gb)$. Indeed, let
\[
\begin{cases}
\omega_1:=\psi_1\in\cs_{\Delta_+^\perp}(\gb)\\
\omega_{n+1}:=\omega_n\times\psi_{n+1}\in\cs_{\Delta_+^\perp}(\ga_{n+1}),\quad n\in\bn
\end{cases}
\]
where $\ga_{n+1}=\ga_n\ok_{\min}\gb=\ga_n\ok_{+}\gb$ (the last equality being given by the nuclearity of $\gb$). Notice that the sequence $(\omega_n)_{n\in\bn}$ evidently satisfies the relations
\[
\omega_{n+1}\circ\iota_n=\omega_n,\quad n\in\bn
\]
so that we can define a (algebraically) positive, unital, linear functional $\omega_\infty\colon\ga_\infty\to\bc$ by
\[
\omega_\infty(a):=\omega_n(a_n)
\]
for every $a\in\ga_\infty$, $n\in\bn$ and $a_n\in\ga_n$ that satisfy $\phi_n(a_n)=a$. Moreover, $|\omega_\infty(a)|\leq\|a_n\|_{\ga_n}=\|a\|_{\ga}$ hence $\omega_\infty$ extends to a well-defined state $\omega$ on $\ga$, the unique one satisfying
\[
\omega\left(j_1(b_1)\dots j_n(b_n)\right)=\prod\limits\limits_{i=1}^n\psi_i(b_i)
\]
for every $b_i\in\gb_i=\gb$, $i=1,\dots,n$ ($n\in\bn$). Plus, $\omega$ is invariant under the restriction to $\Delta_+^\perp$ of the $K_4$-action on $\ga$. By denoting $\omega$ with $\prod\limits_{n\in\bn}\psi_n\in\cs_{\Delta_+^\perp}(\ga)$, we collect the result of the above construction in the following
\begin{cor}
Let $(\gb,K_4,\beta)$ be a $C^*$-dynamical system, with $\gb$ nuclear. If $(\psi_n)_{n\in\bn}\subset\cs_{\Delta_+^\perp}(\gb)$, then the infinite product functional $\omega:=\prod\limits_{n\in\bn}\psi_n$ is a well-defined, $\Delta_+^\perp$-invariant state of the (minimal) Klein twisted chain $(\ga,K_4,\alpha)$ of $\gb$. 
\end{cor}
Of crucial importance for the De Finetti theorem on a Klein twisted $C^*$-chain are, again, the infinite product states, now of the form $\prod\limits_{n\in\bn}\psi$ for some fixed $\psi\in\cs_{\Delta_+^\perp}(\gb)$. As expected, they will be exactly the ergodic symmetric states of the case in question. We follow the path traced in \cite{St2}, as already done in \cite{Ffe} in the Fermi case.
\begin{lem}\label{weaklimit}
Let $\om\in\ce_\bp(\ga)$. Then, $
\wlim\limits_{n\to+\infty}(\pi_\om(\rho_n(a))\xi_\omega)=\om(a)\xi_\omega\quad(a\in\ga)$.
\end{lem}
\begin{thm}\label{extremalvin}
Let $(\gb,K_4,\beta)$ be a $C^*$-dynamical system, with $\gb$ nuclear, and $\ga$ its associated Klein chain. If $\omega\in\cs_\bp(\ga)$, the following are equivalent:
\begin{itemize}
\item[(i)] $\omega\in\ce_\bp(\ga)$
\item[(ii)] $\omega$ is strongly clustering
\item[(iii)] $\omega$ is weakly clustering in average
\item[(iv)] $\omega=\prod\limits_{n\in\bn}\psi$ for some $\psi\in\cs_{\Delta_+^\perp}(\gb)\sim\cs(\gb_{(0,0)}\oplus\gb_{(1,1)})$
\end{itemize}
In particular, $\ce_\bp(\ga)=\left\{\prod\limits_{n\in\bn}\psi\right\}_{\psi\in\cs_{\Delta_+^\perp}(\gb)}$ is weakly-$^*$ closed, thus making $\cs_\bp(\ga)$ a Bauer simplex in $\cs(\ga)$. Precisely, $(\cs_{\Delta_+^\perp}(\gb),\tau_{w^*})$ and $(\ce_\bp(\ga),\tau_{w^*})$ are homeomorphic via the mapping
\[
\begin{aligned}
\iota\colon\cs_{\Delta_+^\perp}(\gb)&\to\ce_\bp(\ga)\\
\psi&\to\prod\limits_{n\in\bn}\psi
\end{aligned}
\]
Lastly, there exists a $\bp$-invariant, densely ranged p.u. map of $C^*$-algebras \[
T\colon\ga\to C(\cs_{\Delta_+^\perp}(\gb))
\]
s.t. its transpose $T^{\text{t}}\colon\cam_1(\cs_{\Delta_+^\perp}(\gb))\to\cs_\bp(\ga)$ is an affine homeomorphism from the probability measures on $\cs_{\Delta_+^\perp}(\gb)$ onto the symmetric states on $\ga$.
\end{thm}
We are now in position to establish the De Finetti theorem for Klein twisted $C^*$-chains, thus obtaining that any symmetric state is the mixture of product states, being each of them the product of a single $\Delta_+^\perp$-invariant state.
\begin{thm}[De Finetti theorem for Klein twisted $C^*$-chains]
\label{defklein}
Let $(\ga,\bp)$ be the $C^*$-dynamical system associated to a unital, $K_4$-graded, nuclear $C^*$-algebra $\gb$. Then, for each $\varphi\in\cs_\bp(\ga)$, there exists a unique $\prec$-maximal $\mu_\varphi\in\cam_1(\cs_\bp(\ga))$ s.t.
\begin{equation}\label{bariklein1}
\varphi(a)=\int\limits_{\cs_\bp(\ga)}\omega(a)\,\mathrm{d}\mu_\varphi(\omega),\qquad a\in\ga\,.
\end{equation}
In particular, $\mu_\varphi$ is pseudo-supported by $\ce_\bp(\ga)=\left\{\prod\limits_{n\in\bn}\psi\right\}_{\psi\in\cs_{\Delta_+^\perp}(\gb)}$ i.e. $\mu_\varphi(B)=1$ for every $B\in\cb_0(\cs_\bp(\ga))$ containing $\ce_\bp(\ga)$. The relative weak-$^*$ topology on the unit ball $B_{\ga^*}$ of $\ga^*$ is metrizable if and only if $\gb$ is separable, in which case $\mu_\varphi$ is supported by $\ce_\bp(\ga)$ and \eqref{bariklein1} becomes
\begin{equation*}
\varphi(a)=\int\limits_{\ce_\bp(\ga)}\omega(a)\,\mathrm{d}\mu_\varphi(\omega),\qquad a\in\ga\,.
\end{equation*}
\end{thm}

\section{Concluding remarks}
\label{corem}

In the present final section, we collect some results of interest which can be obtained by using the lines in \cite{St2}, and then in \cite{Ffe} (1st part, leaving the details to the reader), as well as interesting open problems (2nd and 3rd parts) on which we plan to come back in the future. In the first part, we are tacitly assuming the setting of Section \ref{klsce} and remark that the results listed there hold true whenever the set of symmetric states of any other model is a Choquet simplex and were already proved in \cite{St2} (for the usual tensor product) and \cite{Ffe} (for the Fermi Systems),
and might be applied also to some other interesting situations such as those described in the 3rd part. 

\medskip

\noindent
\textbf{On the simplex of symmetric states.} Recall that a {\it face} of the simplex $\cs$ is a convex subset $F\subset\cs$ satisfying the following hereditary property: 
if $\om\in F$ dominates $\f\in\cs$ ({\it i.e.} $\f\leq t\om$ for some $t>0$), then also $\f\in F$. 

For any $C^*$-algebra $\ga$ and for a state $\om\in\cs(\ga)$, recall that the support $s(\om)\in\ga^{**}$ of its extension to the bidual is central (i.e. $s(\om)\in Z\big(\ga^{**}\big)$) if and only if the cyclic GNS vector $\xi_\om$ is also cyclic for $\pi_\om(\ga)'$.

Suppose that $\cam\subset\cb(\ch)$ is a von Neumann algebra acting on the separable Hilbert space $\ch$.
Consider the central decomposition
$\cam=\int^\oplus_\G \cam_\g\di\n(\g)$, see {\it e.g. }\cite{S}. Here, $\G$ is a locally compact open dense subset of $\s({\rm Z}(\cam))$, and $\n$ is a Radon probability measure with ${\rm supp}(\n)=\G$. By Theorem 21.2 of \cite{N}, if $\la\in[0,1]$, the set 
$$
E_\la:=\{\g\in\G\mid \cam_\g\,\text{is a factor of type}\,\, \ty{III}{\,\!}_\la\}
$$
is measurable. For von Neumann algebras acting on separable Hilbert spaces, it is then meaningful to speak about {\it von Neumann algebras of type} $\ty{III}_\la$ as those for which $\n(E_\la)=1$.
By $p_\la\in Z(\cam)$ as the central selfadjoint projection associated to indicator function $\chi_{E_\la}$, 
$p_\la\cam=\int^\oplus_{E_\la}\cam_\g\di\n(\g)$ 
is a $\ty{III}_\la$ von Neumann algebra acting on the Hilbert space $p_\la\ch$.

The infinite $C^*$-tensor product chain based on the $C^*$-dynamical system $(\gb,G,\b)$ as in Section \ref{sevnif}, based on the hermitian bicharacter $u$ (typically those listed in Theorem \ref{three}) and using the min-norm, 
 is here denoted simply by $\ga$.

For $\#\in\{\text{I}, \ty{II_1},\ty{II_\infty}, \ty{III}\}$ and $\la\in[0,1]$, let
$$
\cs_\bp(\ga)_\#:=\{\om\in\cs_\bp(\ga)\mid \pi_\om(\ga)''\,\,\text{is of type}\,\,\#\}
$$
and, when $\gb$ is separable,
$$
\cs_\bp(\ga)_\la:=\{\om\in\cs_\bp(\ga)\mid \pi_\om(\ga)''\,\,\text{is of type}\,\,\la\}\,,
$$
respectively.

\begin{rem}
\label{face}
By following the lines in \cite{St2, Ffe}, it is possible to prove the following assertions:
\begin{itemize}
\item[{\bf-}\!{\bf-}] for each $\#$, $\cs_\bp(\ga_{\rm F})_\#$ is a face of $\cs_\bp(\ga_{\rm F})$;
\item[{\bf-}\!{\bf-}] $\cs_\bp(\ga)$ is the convex hull of its faces
$\cs_\bp(\ga)_\#$,  $\#=\text{I}, \ty{II_1},\ty{II_\infty}, \ty{III}$;
\item[{\bf-}\!{\bf-}] $\cs_\bp(\ga)_\la$ is a face of $\cs_\bp(\ga)$, provided $\gb$ is separable.
\end{itemize} 
\end{rem}
Now we pass to a class of results, analogous to the analogous ones in \cite{St2, Ffe}, whose direct proof requires the use of the Klein transformation, outlined in Section \ref{bptwtep}. Therefore, for the rest of the present subsection, we assume that the action $\b$ is inner (i.e. implemented by the inner action of unitaries of $\gb$). 

If $\D_+\equiv\D_{+,u}\subset\widehat{G}$ is the isotropy subgroup defined in \eqref{annis}, for a given state $\f\in\cs_{\D_{+}^\perp}(\gb)$ denote by $\om$ its infinite product state acting on $\ga$.
\begin{prop} [\cite{Ffe}, Proposition 10.1]
\label{stscu}
The infinite product state $\om$
on $\ga$ has central support in the bidual if and only if $\f$ has. 
\end{prop}

The type of factors generated by product states is described in the following
\begin{prop}[\cite{St2}, Thm. 2.2]
\label{stfst}
For $\f\in\cs_{G_o}(\gb)$, the product state $\om\in\cs_\bp(\ga)$ is a factor state if and only if $\f$ is. If this is the case,
\begin{itemize}
\item[(a)] $\om$ is of type $\ty{I}_1$ if and only if $\f$ is a homomorphism;
\item[(b)] $\om$ is of type $\ty{I}_\infty$ if and only if $\f$ is pure, but not  a homomorphism;
\item[(c)] $\om$ is of type $\ty{II}_1$ if and only if $\f$ is trace, but not a homomorphism;
\item[(d)] $\om$ is of type $\ty{II}_\infty$ if and only if $\om_{\xi_\f}\lceil_{\pi_\f(\gb)'}$ is trace, and $\f$ is neither pure nor a trace;
\item[(e)] $\om$ is of type $\ty{III}$ if and only if $\om_{\xi_\f}\lceil_{\pi_\f(\gb)'}$ is not a trace.
\end{itemize}
\end{prop}
The proof of the last two propositions relies on the following fact: the Klein transformation provides the unitary equivalence between the GNS representation of product state $\om$ on $\ga$ and the corresponding one of the product state $\psi_{\f,\f\dots}$ on
$\otimes_{\bn,\min}\gb$. Now it is enough to reason as in \cite{St2}.

\medskip

\noindent
\textbf{Symmetric and spreadable states on the infinite noncommutative torus.} 
The noncommutative torus is one of the most relevant examples of twisted tensor product. It is therefore meaningful to investigate the {\it infinite dimensional noncommutative torus}, which is nothing but the twisted tensor product of infinitely many copies of the algebra $C(\bt)$ by using the bicharacters
\begin{equation}
\label{bkyu}
u_\th(m,n):=e^{2\pi\imath\th mn}\,,\quad m,n\in\bz\,,\,\theta\in\br.
\end{equation}
Since $C(\bt)$ is nuclear, there is no ambiguity concerning the $C^*$-norm used in constructing such a model, denoted by
\[
\bb_\th:=C(\bt)\ot C(\bt)\ot\dots\ot C(\bt)\dots
\]
where $u$ stands for $u_\theta$. Since the bicharacter in \eqref{bkyu} is not hermitian in general, the investigation of symmetric states cannot be directly developed on the chain $\bb_\th$. On the other hand, the properties of this class of states may be well studied following the approach in Section \ref{free}. To simplify the matter (i.e. to treat models up to $*$-isomorphisms), we reduce the analysis to $\theta\in(0,1/2)$ since the case $\th=0$ is the commutative one, whereas $\th=1/2$ corresponds to the Fermi case (see Section \ref{ratogato}) studied in \cite{Ffe}. We also briefly recall the concept of {\it spreadability}, concerning states which are invariant after ``passing to subsequences'', see e.g. \cite{CFG}.

The question relative to the symmetric and spreadable states of the noncommutative torus was fully solved in \cite{ADR}, as reported here below.
\begin{prop}[\cite{ADR}, Proposition 3.9 (ii)]
The class of symmetric states of $\bb_\th$, $0<\th<1/2$ is empty.
\end{prop}

The following comments are now in order. 
\begin{itemize}
\item[$1^o$.] The fact that $\cs_\bp(\bb_\th)=\emptyset$ appears quite surprising, as
it does not contain the tracial state, that is the infinite product of the Haar-Lebesgue state $\int_\bt(\,\,\,)\frac{\di\th}{2\pi}$. After a direct computation, it is not hard to convince yourself that the condition for the existence of symmetric states of $\bb_\th$, $0\leq\th\leq1/2$, is the hermiticity of the involved bicharacter \eqref{bkyu}. This can only happen when $\th=0,1/2$, in which cases the study of symmetric states can be carried out directly on $\bb_\th$, as noticed above.
\item[$2^o$.] Concerning spreadable states, Theorem 3.8 in \cite{CDGR} proves that
\begin{itemize}
\item[{\bf-}\!{\bf-}] if $\th$ is irrational, the set of spreadable states is the singleton made of the canonical trace on $\bb_\th$;
\item[{\bf-}\!{\bf-}] if $\th=m/n$ with $\gcd(m,n)=1$, such a set has the form
$$
\om=\int_{\cs_{\Delta_+^\perp}(C(\bt))}(\psi\times\psi\cdots\times\psi\times\cdots)\di\n(\psi)
$$
for a Radon probability measure on $\cs_{\Delta_+^\perp}\big(C(\bt)\big)$. 
\end{itemize}
Therefore, (the part of) Ryll-Nardzewski Theorem, that is that relative to the coincidence between the sets of spreadable and symmetric states, fails for all $\th\in(0,1/2)$ and holds true for $\th=0$ (classical case), and $\th=1/2$ (Fermi case). It should be noted that the results in \cite{CDGR} were obtained by exploiting the ``isotropy group'', firstly introduced and studied in Section II.2 of \cite{V}, and reported here in Section \ref{visot}.
\end{itemize}

In the spirit of our work, we note that the CAR algebra on infinitely many generators is nothing but the Fermi chain where the sample algebra is $\gb=\bm_2(\bc)$. Now, we mention \cite{CDR} in which, following the lines of Ryll-Nardzewski Theorem, a comparison between the class of symmetric states (firstly studied in \cite{CF}, and reconsidered in Theorem 6.6 of \cite{FMC} with a different approach), spreadable and rotatable states of the CAR algebra is also done.

Lastly, Theorems 3.2 and 4.10 in \cite{ADR} are connected with our analysis of symmetric states. The first one (Theorem 3.2) establishes a necessary and sufficient condition for which the class of symmetric states and that of spreadable ones coincide. Such a condition is satisfied if the involved bicharacter is hermitian (``anti-symmetric" in their terminology). Under the condition of nuclearity of the sample algebra, Theorem 4.10, instead provides two conditions for which the class of spreadable states forms a Choquet (in particular, Bauer) simplex, whose extremal points are made of infinite products of a single state, invariant under the action of the annihilator of the isotropy group. Such conditions are
\begin{equation}
\label{ittezznorts}
\text{either}\,\,u(\s,\t)=1\,,\,\, \text{or}\,\,u(\s,\t)u(\t,\s)\neq1\,;\, \s,\t\in \D_+\,.
\end{equation}
Therefore, since in our setting Theorem 3.2 corresponds to the situation when $u$ is hermitian (cf. Section \ref{free}), only the first condition in \eqref{ittezznorts} is satisfied. Theorem 4.10 is then precisely what we did (cf. \cite{V}): it is the content of our Section \ref{ssttptep} (see also Theorem \ref{defklein}).

Differently by \cite{ADR}, in which any possible choice of a bicharacter is allowed, in our investigation we decided to deal with nondegenerate bicharacters in order to produce canonical models, see Proposition \ref{nondeg}. This approach provided the interesting quite surprising fact that a fruitful investigation of symmetric states can be developed in three cases only (the usual one, i.e. the Bose case, the Fermi case, and the completely new case associated to the Klein group), as one of us proved in Section III.7 of \cite{V}, and reported here in Section \ref{visot}. Our Section \ref{klsce}, devoted to the complete analysis of the Klein case (not mentioned in \cite{ADR}), is nothing but the combined consequence of the results listed above. 

As our Theorem \ref{defklein}, Theorem 4.10 in \cite{ADR} assumes nuclearity of the sample algebra. This implies the existence of a unique compatible $C^*$-norm on the algebraic twisted tensor product, see Theorem 6.4 in \cite{FV1}. If for the usual tensor product the construction of the infinite chain is quite elementary for any choice of a $C^*$-cross norm, in presence of a nontrivial bicharacter this is a delicate question, as well as the construction of the action of the permutation group $\bp$ (cf. Section \ref{opytg}). The basic requirement concerning the construction of the twisted infinite chain in order to investigate the structure of spreadable states is ignored in \cite{ADR}, whereas we thoroughly explain them in Sections \ref{asszac} and \ref{sevnif}.

\medskip

\noindent
\textbf{Symmetric states on various completions of the twisted tensor product.} 
To the best of the authors’ knowledge, the first results relative to the structure of symmetric states in noncommutative setting, and the relative De Finetti Theorem are provided in two contemporary and independent works \cite{HP, St2} dealing with the usual (i.e. untwisted) tensor product. Such results seem to be different because the former is referred to the completion of the infinite chain based on the max-norm, whereas the latter concerns the min-norm. 

To explain what are the kind of problems, we consider two $C^*$-algebras $\ga_1$ and $\ga_2$, supposed to be unital for simplicity, together with a unital positive linear map $\Psi:\ga_1\to\ga_2$ of $\ga_1$ onto $\ga_2$. It is clear that $\Psi^{\rm t}:\cs(\ga_2)\to\cs(\ga_1)$ is injective and maps states into states.
For $i=1,2$, consider two convex, $*$-weakly compact subsets $\cs_i\subset\cs(\ga_i)$. It is unclear under which conditions $\Psi^{\rm t}(\cs_2)=\cs_1$. A pivotal example is when
a group $G$ is acting on $\ga_1$ by $*$-automorphisms $G\stackrel{\a}{\curvearrowright}\ga_1$. In this situation, with $\ga_2=\ga_1^{G}$, $\Psi$ the conditional expectation onto the fixed-point subalgebra $\ga_1^{G}$, and $\cs_i=\cs(\ga_i)$, $i=1,2$. It is clear that $\Psi^{\rm t}(\cs_2)\subsetneq\cs_1$ in all nontrivial cases: the invariant states are, in general, a strict subset of all states of $\ga_1$.

By coming back to the situation of interest, we fix a unital $C^*$-algebra $\gb$ and consider the infinite one-sided chains $\ga_1$ and $\ga_2$ constructed using the max and the min-norm, respectively.
We notice that there is a $*$-epimorphism $\pi:\ga_1\to\ga_2$. Since the group of all finite permutations $\bp$ acts on both algebras, it is then meaningful to consider the corresponding sets of symmetric states $\cs_1$ and $\cs_2$, respectively. From Theorem 5.4 in \cite{HP} and Theorem 2.7 in \cite{St2}, it results that $\pi^{\rm t}(\cs_2)=\cs_1$ that is, roughly speaking, the set of symmetric states are the same for both infinite tensor product chains obtained by using the max or the min-norm, indifferently.

It is a very delicate question to understand whether an analogous result holds true in the case of the one-sided twisted chains constructed by using any, possibly compatible, $C^*$-norm. This is precisely the motivation for which we assumed nuclearity in Theorem \ref{extremalvin}. 

Concerning the case of interest in the present paper, including the simplest nontrivial case relative to Fermi systems, we consider a fixed $C^*$-dynamical system $(\gb, G,\a)$, being $u:\widehat{G}\times\widehat{G}\to\bt$ a hermitian bicharacter. Taking into account that Lemma \ref{associativity} seems not to depend on the chosen $C^*$-norm and, after noticing that Theorem 8.1 in \cite{FV1} would allow the extension of the flip as a $*$-automorphism of $\gb\ot_{\rm max}\gb$ as well, one might also construct the one-sided infinite chain tensor product $\ga^{(u)}_{\rm max}$ on which $\bp$ is acting. Now, in the previous picture, $\ga_1\equiv\ga^{(u)}_{\rm max}$ and $\ga_2\equiv\ga^{(u)}_{\rm min}$ (where the latter algebra was simply denoted by $\ga^{(u)}$ above).
The combined consideration of both results in \cite{HP, St2} would produce $\pi^{\rm t}\big(\cs\big(\ga^{(u)}_{\rm min}\big)\big)=\cs\big(\ga^{(u)}_{\rm max}\big)$.
Showing the analogous equality of symmetric states by replacing the max-norm with any other $C^*$-norm $\g\geq \|\,\,\,\|_{\rm min}$, as happens in Section \ref{klsce},
might be more complicated (and, perhaps, possibly not true).

\section*{Declarations}
\noindent
{\bf Conflict of interest.} On behalf of all authors, the corresponding author states that there is no conflict of interest.\\
{\bf Data availability.} On behalf of all authors, the corresponding author states the data availability.

\section*{Acknowledgements}
The first author acknowledges ``Excellence Department Project'', CUP E83C23000330006; and ``Tor Vergata University of Rome funding OANGQS'', CUP E83C25000580005. Furthermore, the authors would like to thank two anonymous referees for the valuable comments and suggestions, which have contributed to improving the clarity and presentation of the work.


\begin{thebibliography}{9999} 

\bibitem{ADR} Aiello V., Del Vecchio S., Rossi S. {\it On the Ryll-Nardzewski theorem for quantum stochastic processes}, Adv. Math. {\bf 501} (2026), 111119.

\bibitem{AM1} Araki H., Moriya H.
{\it Equilibrium statistical mechanics of Fermion lattice systems},
Rev. Math. Phys. {\bf 15} (2003), 93--198.

\bibitem{Bat} Batty C.J.K.
{\it Simplexes of states of {$C^*$}-algebras},
J. Operator Theory {\bf 4} (1980), 3--23.

\bibitem{Bo} Boca F.-P. {\it Rotation $C^*$-algebras and almost Mathieu operators}, Theta, Bucharest, 2001.

\bibitem{BR} Bratteli O., Robinson D. W.
{\it Operator algebras and quantum statistical mechanics I, II},
Springer, Berlin--Heidelberg--New York, 2001, 2002.

\bibitem{CDGR} Crismale V., Del Vecchio S., Griseta M. E., Rossi S. {\it De Finetti theorem on the infinite non-commutative torus}, Proc. Amer. Math. Soc. {\bf 154} (2026), 723-735.

\bibitem{CDR} Crismale V, Del Vecchio S., Rossi S. {\it On Distributional Symmetries on the CAR Algebra}, Commun. Math. Phys. {\bf 406} (2025) art. n. 276.

\bibitem{CF} Crismale V., Fidaleo F. {\it De Finetti Theorem on the CAR algebra},  Commun. Math. Phys. {\bf 315} (2012), 135-152. 

\bibitem{CF2} Crismale V., Fidaleo F. {\it Symmetries and ergodic properties in quantum probability},  Coll. Math. {\bf 149} (2017), 1-20. 

\bibitem{CFG} Crismale V., Fidaleo F., Griseta M. E. {\it Spreadability for Quantum Stochastic Processes,
with an application to Boolean Commutation Relations},  Entropy {\bf 22} (2020), art. n. 532. 

\bibitem{DeF} De Finetti B.
{\it Funzione caratteristica di un fenomeno aleatorio}, Atti
Accad. Naz. Lincei, VI Ser., Mem. Cl. Sci. Fis. Mat. Nat. {\bf 4}
(1931), 251--259.

\bibitem{Ffe} Fidaleo F.
{\it Symmetric states for $C^*$-Fermi systems}, Rev. Math. Phys. {\bf 33} (2022), 2250030 (39 pages).

\bibitem{FMC} Fidaleo F.
{\it On Fermi Quantum Markov Chains}, Internat. J. Theor. Phys. {\bf 64} (2025), 331 (21 pages).

\bibitem{FV} Fidaleo F., Vincenzi E. {Graded $C^*$-algebras and twisted $C^*$-tensor products}, Ricerche Mat. {\bf 74} (2025),
163-184.

\bibitem{FV1} Fidaleo F., Vincenzi E. {$C^*$-norms on the twisted tensor product and nuclearity}, Ricerche Mat. {\bf 74} (2025), 2927-2948.
	
\bibitem{HS} Hewitt E., Savage L. F.
{\it Symmetric measures on Cartesian products}, Trans. Amer. Math.
Soc. {\bf 80} (1955), 470--501.

\bibitem{HP} Hulanicki A., Phelps R. R. {\it Some applications
of tensor products of partially-ordered Linear Spaces}, J. Funct. Anal. {\bf 2} (1968), 177-201.

\bibitem{Mar} Marcinek W.
{\it Particles and quantum symmetries}, Rep. Math. Phys.
{\bf 43} (1999), 239-245.

\bibitem{N} Nielsen O. A.
{\it Direct integral theory}, Marcel Dekker, New York-Basel, 1980.

\bibitem{RW} Rittenberg V., Wyler D.
{\it $\bz_2\oplus\bz_2$-graded Lie algebras and superalgebras}, J. Math. Phys.
{\bf 19} (1978), 2193-2200.

\bibitem{S} Sakai S.
{\it $C^*$-algebras and $W^*$-algebras}, Springer, Berlin--Heidelberg--New
York 1971.

\bibitem{Sch} Scheunert M.
{\it Generalized Lie algebras and superalgebras}, J. Math. Phys.
{\bf 20} (1979), 712-720.

\bibitem{Sta} Stacey P. J.
{\it An action of the Klein four group on the irrational rotation $C^*$-algebra}, Bull. Austral. Math. Soc.
{\bf 56} (1997), 135-148.

\bibitem{St2} St{\o}rmer E.
{\it Symmetric states of infinite tensor products of $C^{*}$--algebras},
J. Funct. Anal. {\bf 3} (1969), 48--68.

\bibitem{SVJ} Stoilova N. I., Van der Jeugt J.
{\it $\bz_2\times\bz_2$-graded Lie (super)algebras
and generalized quantum statistics}, Int. J. Geom. Methods Mod. Phys.
{\bf 23} (2026), 2540028 (17 pages).

\bibitem{T} Takesaki M.
{\it Theory of operator algebras I,III}, Springer, Berlin--Heidelberg--New
York 2002, 2003.

\bibitem{To} Tolstoy V. N.
{\it Once more on parastatistics}, Phys. Part. Nucl. Lett.
{\bf 11} (2014), 933-937.

\bibitem{V} Vincenzi E. {\it $C^{*}$-dynamical systems and Ergodic Theory. Quantum decoherence, twisted tensor products and De Finetti theorem}, Ph.D. thesis, University of Tor Vergata (June 2024), https://www.mat.uniroma2.it/dottorato/Theses/2024-25/Vincenzi\%20Elia.pdf

\bibitem{WO} N. E. Wegge-Olsen
{\it K-Theory and $C^*$-Algebras}, Oxford University Press, (1993).


\end{thebibliography}
\end{document}